\documentclass[12pt,a4paper]{amsart}
\usepackage{amsfonts}
\usepackage{amsthm}
\usepackage{amsmath}
\usepackage{amscd}
\usepackage[latin2]{inputenc}
\usepackage{t1enc}
\usepackage[mathscr]{eucal}
\usepackage{indentfirst}
\usepackage{graphicx}
\usepackage{graphics}
\usepackage{pict2e}
\usepackage{epic}

\numberwithin{equation}{section}
\usepackage[margin=2.9cm]{geometry}
\usepackage{epstopdf} 
\usepackage{mathtools}
\usepackage{commath}
\usepackage{amssymb}  
\usepackage{float}
\usepackage{amstext} 
\usepackage{array}   

\usepackage{hyperref}

\usepackage{cleveref} 
\newcolumntype{L}{>{$}l<{$}} 

\usepackage{tikz-cd}

\usepackage{xcolor}

\definecolor{defcolour}{rgb}{0.6,0.3,1}

\definecolor{MC1}{HTML}{4ed1c8}

\theoremstyle{plain}
\newtheorem{Th}{Theorem}[section]
\newtheorem{lem}[Th]{Lemma}
\newtheorem{Cor}[Th]{Corollary}
\newtheorem{Prop}[Th]{Proposition}

\theoremstyle{definition}
\newtheorem{Def}[Th]{Definition}

\newtheorem{Rem}[Th]{Remark}
\newtheorem{?}[Th]{Problem}
\newtheorem{Ex}[Th]{Example}
\newtheorem{notation}[Th]{Notation}

\title[Gluing dimer quivers]{A gluing operation for dimer quivers}

\author{Karin Baur}
\email{ka.baur@me.com}
\author{Colin Krawchuk}
\email{ctgk2@dpmms.cam.ac.uk}

\begin{document}

\begin{abstract}
In this article we introduce a gluing operation on dimer models. This allows us to construct dimer quivers on arbitrary surfaces. 
We study how the associated dimer and boundary algebras behave under the gluing and how to determine them from the gluing components. 
We also use this operation to construct homogeneous dimer quivers on annuli. 
\end{abstract}

\maketitle

\thispagestyle{plain} 


\section*{Introduction}\label{sec:introduction}

Dimer models on surfaces have been studied in different contexts, e.g. in statistical mechanics, \cite{Temperley1961DimerPI}, integrable systems, \cite{GK13}, in mirror symmetry, e.g.~\cite{Bocklandt16}, to cite a few. 
More recently, dimer models on surfaces with boundary have been considered, cf.~\cite{BKM, franco2014bipartite}. In particular, dimer models on the disk appear in the context of cluster categories associated with the Grassmannian~\cite{BKM} or~\cite{CKP}. 

The paper is organised as follows: 
Sections~\ref{sec:dimer-quiver} and~\ref{sec:postnikov} contain the necessary background on dimer quivers, on Postnikov diagrams and their associated algebras. 
In Section~\ref{sec:surfaces}, we recall the notion of weak Postnikov diagrams and how they are linked to dimer quivers on surfaces. 

In Section~\ref{sec:MV}, we introduce a gluing operation on dimer models and study its properties. The purpose of this gluing is to obtain dimer models on general surfaces from those on ``simpler'' surfaces. 
We answer questions about the structure of the created dimer and boundary algebras through the gluing and the component dimer models.

In particular, we show that the dimer algebras obtained through gluing are determined by the components used in the gluing (Proposition~\ref{prop:algebra-glued}). Additionally, when certain consistency conditions are met we also determine the glued boundary algebra from the components in Theorem~\ref{th:bdy-alg}. As an application we use this to calculate the boundary algebra associated to arbitrary Postnikov diagrams on the disk, 
Corollary~\ref{cor:splitting}. 
In Section~\ref{sec:bridge}, we introduce a family of consistent dimer quivers on the disk, called the $k$-bridge. 
These will be instrumental in the construction of dimer quivers for the annulus in Section~\ref{sec:annulus}. In this final section, we glue bridge quivers with certain homogeneous dimer quivers on the disk, called $(k,n)$-diagrams. We prove that the resulting dimer quivers on the annulus correspond to weak Postnikov diagrams of degree $k$, Theorem~\ref{thm:main}.

\subsection*{Acknowledgments}
The first author is supported by the EPSRC Programme Grant EP/W007509/1 and by the Royal Society Wolfson Award RSWF/R1/180004. 
The authors would like to thank the Isaac Newton Institute for Mathematical Sciences, Cambridge, for support and hospitality during the programme Cluster algebras and representation theory, where work on this paper was undertaken. 
This work was supported by EPSRC grant EP/R014604/1. 
The second author is supported by an NSERC postgraduate scholarship.  
The authors thank Matthew Pressland for helpful suggestions and comments on an earlier version of this article.


\section{Dimer Quivers}\label{sec:dimer-quiver}


A quiver is an oriented graph $Q=(Q_0,Q_1)$, with finite vertex set $Q_0$ and a finite set of arrows (oriented edges) $Q_1$. We denote by  $Q_{cyc}$ the set of oriented cycles in $Q$. 

\begin{Def}\label{def:quiver-faces}
A \textit{quiver with faces} is a quiver $Q=(Q_0,Q_1)$ together with a set of faces $Q_2$ and a map 
$\partial: Q_2 \to Q_{cyc}$ sending a face $F \in Q_2$ to its boundary $\partial F \in Q_{cyc}$. 
\end{Def}

We consider a special class of quivers with faces $Q_2$: these are obtained from gluing a collection of oriented cycles along arrows in a consistent way. They were first formalised in~\cite{BKM} as dimer models with boundary. We will call them \textit{dimer quivers}.

We first recall the notion of the incidence graph of a vertex of $Q$:  
For vertex $i \in Q_0$ of the quiver $Q$, the \textit{incidence graph} of $Q$ at $i$ has as vertices the set of arrows incident with $i$ and contains an edge between vertices $v_{\alpha}$ and $v_{\beta}$ if the path 
$$ \bullet\stackrel{\alpha}{\longrightarrow} i \stackrel{\beta}{\longrightarrow}\bullet $$ is part of the boundary $\partial F$ of a face $F\in Q_2$. 

If $\alpha\in Q$, the \textit{face multiplicity of $\alpha$} is the number of faces such that $\alpha$ belongs to  their boundary.

\begin{Def}\label{def:dimer-quiver}
    A \textit{dimer quiver} is a quiver with faces $(Q_0,Q_1,Q_2)$ where $Q_2$ can be expressed as the disjoint union of two sets $Q_2^+$ and $Q_2^-$ and which satisfies the additional properties: 
\begin{enumerate}
        \item $Q$ has no loops
        \item every arrow $\alpha$ in $Q_1$ appears with face multiplicity one or two. 
        In the latter case, if $F_1$ and $F_2$ are the two faces $F_1,F_2$ such that $\alpha\in\partial F_i$ for $i=1,2$, then $F_1\in Q_2^+$ and $F_2\in Q_2^-$ (or vice versa). 
        \item The incidence graph of every vertex of $Q_0$ is (non-empty and) connected. 
\end{enumerate}
\end{Def}

\begin{Rem}
    Dimer quivers are the dual graphs to certain well-studied bipartite graphs on surfaces called dimer models (see \cite{broomhead2010dimer,Temperley1961DimerPI,hanany2005dimer} for example).  
\end{Rem}

If $\alpha$ has face multiplicity one in the dimer quiver $Q$, we call $\alpha$ a \textit{boundary arrow}, otherwise, $\alpha$ is an \textit{internal arrow}. A vertex $i\in Q_0$ is a \textit{boundary vertex} if it is incident with a boundary arrow. Otherwise, $i$ it is called \textit{internal}. 

\begin{Ex}\label{ex:dimer quiver} The following is an example of a dimer quiver: 
\begin{center}\includegraphics[width=0.9\linewidth]{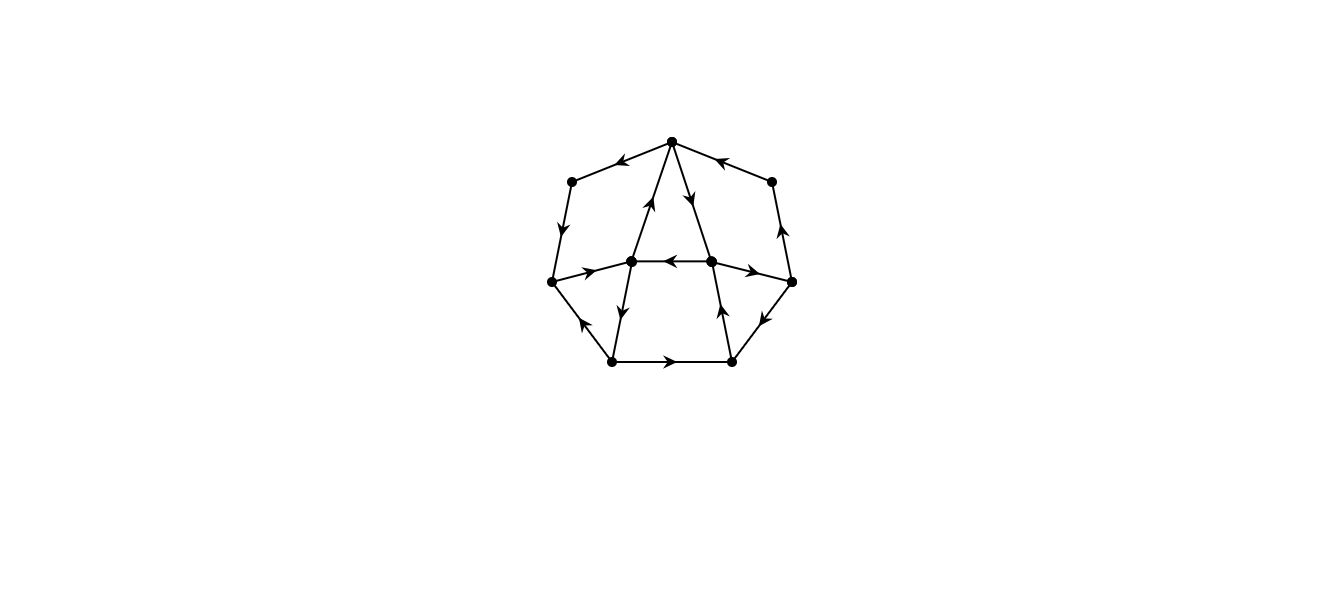}
\end{center}
\end{Ex}

\begin{Rem}\label{rem:dimer-surface}
    Given a dimer quiver $Q$, we can form a topological space $|Q|$ by gluing the arrows bounding each face as directed by their orientation. The resulting space $|Q|$ is an oriented surface with boundary, where the faces in $Q_2^+, Q_2^-$ can be declared to have positive and negative orientations respectively. We will write $\Sigma=\Sigma(Q)$ for the surface $|Q|$ given by the dimer quiver.
$Q$ is said to be a \textit{dimer quiver in a disk}, if $\Sigma_Q=|Q|$ is homeomorphic to a disk.
\end{Rem}

Given any quiver $Q$, let $\mathbb{C}Q$ denote the path algebra over $Q$. This has as a basis the set of all paths of $Q$, including the trivial ones (of length 0), with multiplication given by concatenation of paths. Note that since the dimer quivers we consider contain oriented cycles, the associated path algebras are infinite dimensional. 

For any oriented face $F \in Q_2$, the element $\partial F$ determines a cycle in $Q$, and so the set of all cycles defines an element 
$$W_{Q}=\sum_{F \in Q_{2}^{+}} \partial F-\sum_{F \in Q_{2}^{-}} \partial F $$
in $\mathbb{C}Q_{cyc}$ called the {\em potential} of $Q$. 

Internal arrows define relations on the path algebra, given by the so-called cyclic derivatives of the potential: 
\begin{Def}\label{def:relations-potential}
Let $\alpha$ be an internal arrow of $Q$ and let $\alpha\beta_1\cdots\beta_s=\partial F_1\in Q_{cyc}$ and $\alpha\gamma_1\cdots\gamma_t=\partial F_2\in Q_{cyc}$ be the two boundary cycles of the faces in $Q_2$ containing $\alpha$, with 
$F_1\in Q_2^+$ and $F_2\in Q_2^-$. Let $p_{F_1}:=\beta_1\cdots\beta_s$ and $p_{F_2}:=\gamma_1\cdots\gamma_t$ denote the complements of $\alpha$ in $\partial F_1$ and $\partial F_2$ respectively.
Then the cyclic derivative of $W$ with respect to $\alpha$ is defined as 
$\partial_{\alpha}(W)= p_{F_1} - p_{F_2}$. 
With this, the potential $W$  determines an ideal of 
relations by $I:= \langle \partial_\alpha(W): \alpha \in Q_1 \rangle$ of $\mathbb{C}Q$. 
\end{Def}

\begin{Def}\label{def:dimer-algebra}
Let $Q$ be a dimer quiver, the {\em dimer algebra} $A_Q$ is defined to be the quotient of the path algebra by the 
relations arising from internal arrows $\alpha \in Q_1$.
\end{Def}

Note that since we can take powers of cycles in $Q$ and since they remain after taking the relations from the cyclic derivatives, the algebra $A_Q$ is 
infinite dimensional in general.


As a consequence of the definition, for any vertex $i \in Q_0$ the unit cycles at $i$ (the boundary of a face formed by the product of arrows starting and ending at $i$) are equivalent in $A_Q$. Let $u_i$ denote the corresponding element in $A_Q$.


\begin{lem}\label{lm:central} Let $t$ be the element of $A_Q$ given by $t:= \sum_{i \in Q_0}u_i$, then $\mathbb{C}[t] \subseteq Z(A_Q)$, where $Z(A_Q)$ denotes the centre of $A_Q$.
\end{lem}

\begin{proof}
It suffices to show $t$ commutes with any arrow $\alpha \in Q_1$. Suppose $\alpha $ is an arrow starting at vertex $i$ and terminating at vertex $j$. Then $t\alpha=u_i\alpha$ and $\alpha t=\alpha u_j$, therefore we only have to show $u_i\alpha=\alpha u_j$. Let $F \in Q_2$ be a face whose boundary is given by $\partial F =\alpha p_F$. 
As $\alpha p_F$ is a unit cycle at $i$, we have $u_i\alpha= \alpha p_F \alpha$. Similarly, $p_F \alpha$ is a unit cycle at $j$, and so $\alpha p_F \alpha = \alpha u_j $.
\end{proof}


\begin{Def}\label{def:boundary-alg}
    Let $Q$ be a dimer quiver and $A_Q$ its dimer algebra. Let $e_1,\dots, e_n$ be the idempotent elements of $A_Q$ corresponding to the boundary vertices of $Q$ and set $e=\sum_{i=1}^n e_i$. 
Then the \textit{boundary algebra} of $Q$ is defined to be the idempotent subalgebra given by: 
\[B_Q=eA_Qe.\]
\end{Def}


\section{Postnikov Diagrams}\label{sec:postnikov}


In this section, we consider certain diagrams in the disk:  the surface $\Sigma$ is a disk with $n$ marked points on the boundary. These diagrams correspond to dimer quivers in the disk that exhibit global properties which will later be used to determine the corresponding boundary algebra.


\begin{Def} \label{def:pos-diag} A \textit{Postnikov diagram} (in the disk) consists of oriented curves (strands) in the disk with marked points on its boundary, where each marked point is the source of a strand and the target of exactly one strand. The diagram must satisfy the following axioms.
\begin{enumerate}
    \item There are finitely many intersections, and each intersection is a transversal of order two.
     \item Following a fixed strand, the strands that cross it must alternate between crossing from the left and crossings from the right.
     \item Strands have no self-intersections.
     \item If two strands cross at distinct points, the corresponding bounded region forms an oriented disk.
\end{enumerate}
\end{Def}

It is clear from the definition that a Postnikov diagram determines a permutation of the boundary vertices by $ i \mapsto j$ if the strand starting at vertex $i $ terminates at vertex $j$. We call this strand permutation $\pi_D$. Permutations of the form $i \mapsto i+k \ (mod \ n)$ are of particular interest as they describe a cluster in the homogeneous coordinate ring of the Grassmannian Gr$(k,n)$ ~\cite[Theorem 3]{scott}. 
We refer to such a diagram as a 
$(k,n)$-diagram. 
Note that such diagrams always exist: 
In the paper mentioned above, Scott gave an explicit construction of a $(k,n)$-diagram for every $k<n$. 


\begin{Ex} \label{ex:3-5}
The following is an example of a $(3,5)$-diagram.
\begin{center}
\includegraphics[width=0.85\linewidth]
{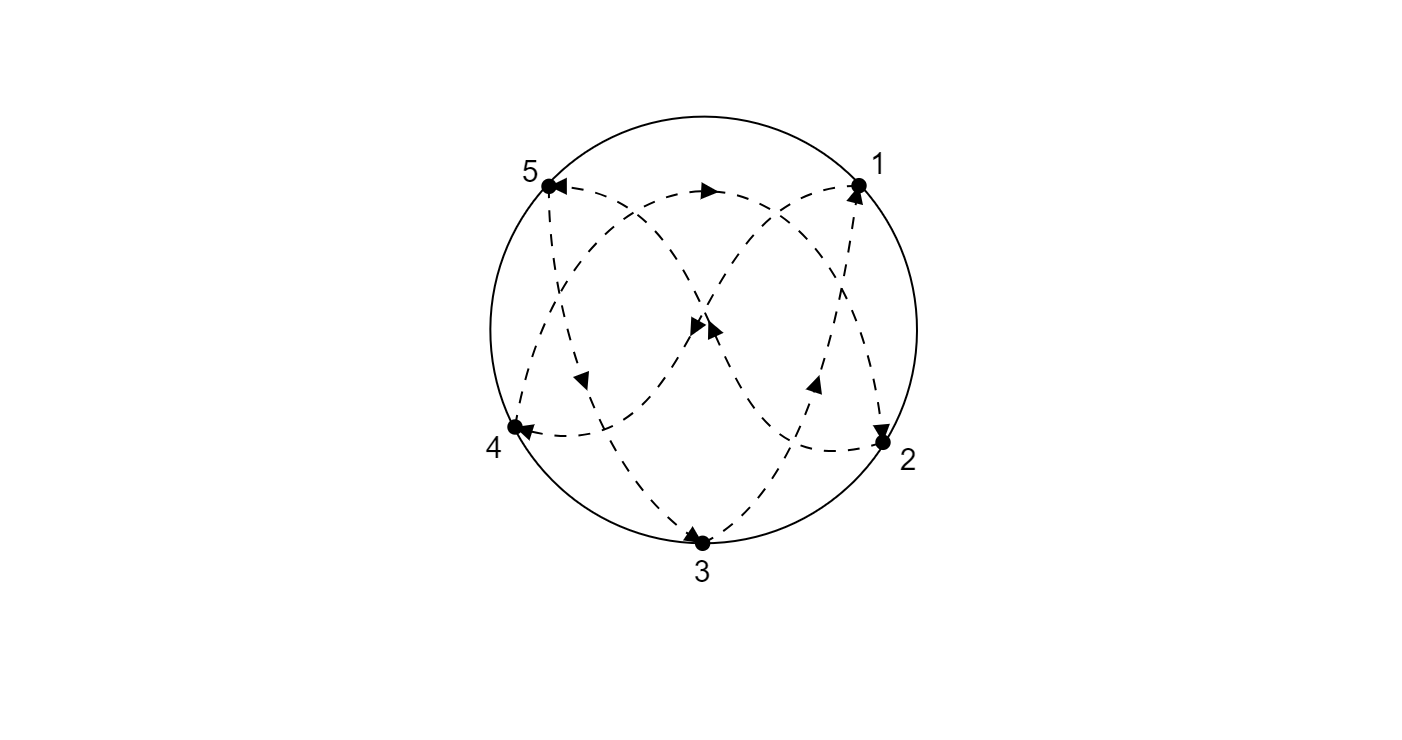}
\end{center}
\end{Ex}


Postnikov diagrams are defined up to isotopy, or deformations of the strand configurations that 
fix endpoints. Two 
diagrams are said to be \textit{equivalent} if one can be obtained from the other through a sequence of local twisting or untwisting moves, which are illustrated in the following diagram. The lower diagram represents the boundary case. 
\begin{center}
\includegraphics[width=0.6\linewidth]{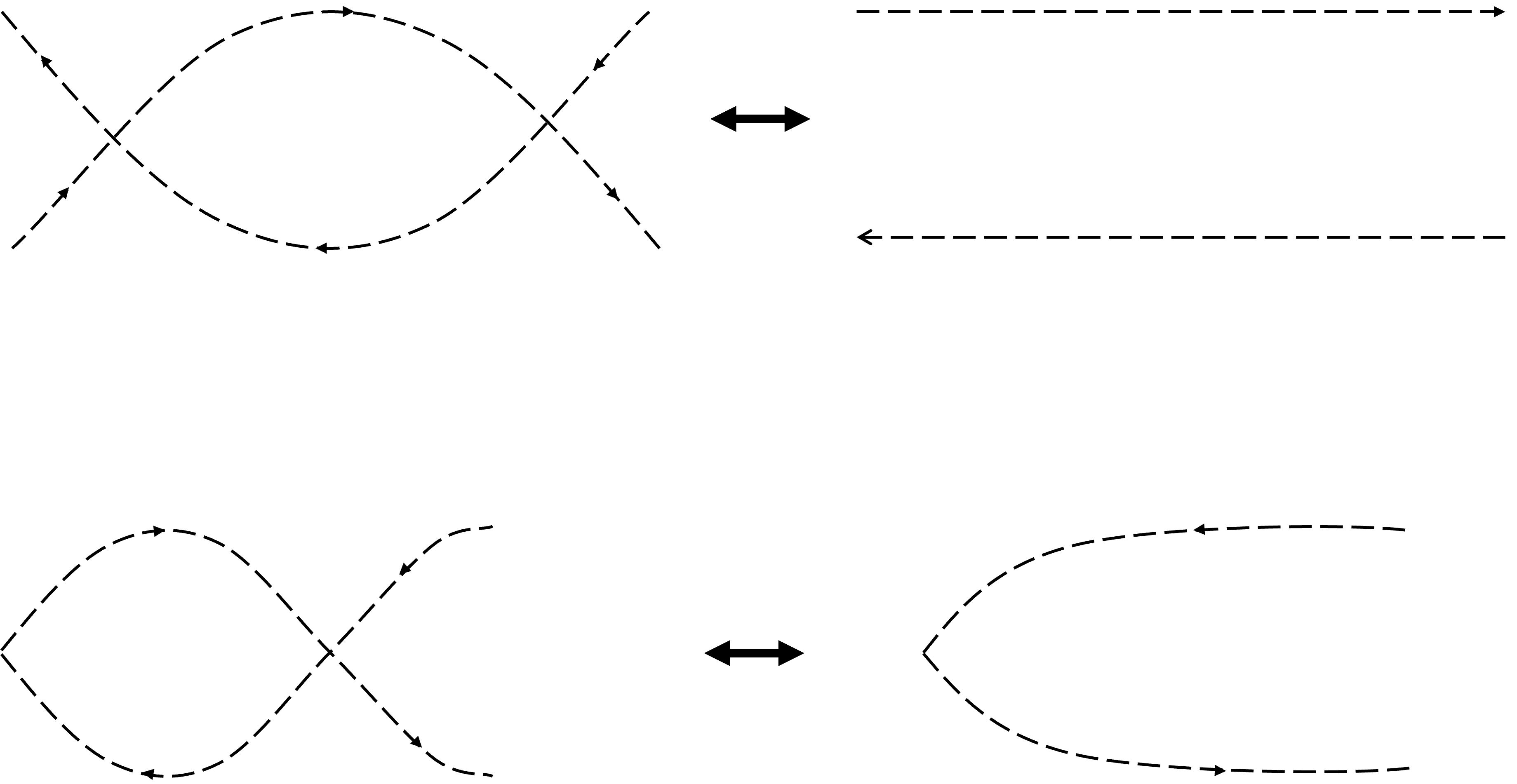}
\end{center}

We call a Postnikov diagram {\em reduced} if it may not be simplified via untwisting moves. The diagram in Example~\ref{ex:3-5} is 
a reduced Postnikov diagram.


Each diagram divides the interior of the disk into regions, which are the connected regions in the complement of the union of all strands. In the definition of a Postnikov diagram we required that strands must alternate between crossing from left to right, and right to left when following a particular strand. As a consequence, the interior regions can be classified as either,
\begin{enumerate}
    \item alternating if the strands forming its boundary alternate in orientation,  or
    \item oriented if the strands are all oriented in the same direction (all clockwise or all anticlockwise).
\end{enumerate}

Notice that each strand splits the disk into a left and right region with respect to the strand orientation. This allows us to attach labels to the alternating regions: An alternating region is labelled by $i \in [n]$ if the strand starting at $i$ keeps the given region on its left.

We want to associate to every Postnikov diagram a dimer quiver. Its vertices are given by the alternating regions. The arrows arise from crossings: 
To each intersection point between strands of alternating regions, we associate an arrow between the corresponding vertices, with the direction described in the following diagram. The figure on the right represents the case where both regions are on the boundary. 

\begin{center}
\includegraphics[width=0.9\linewidth]{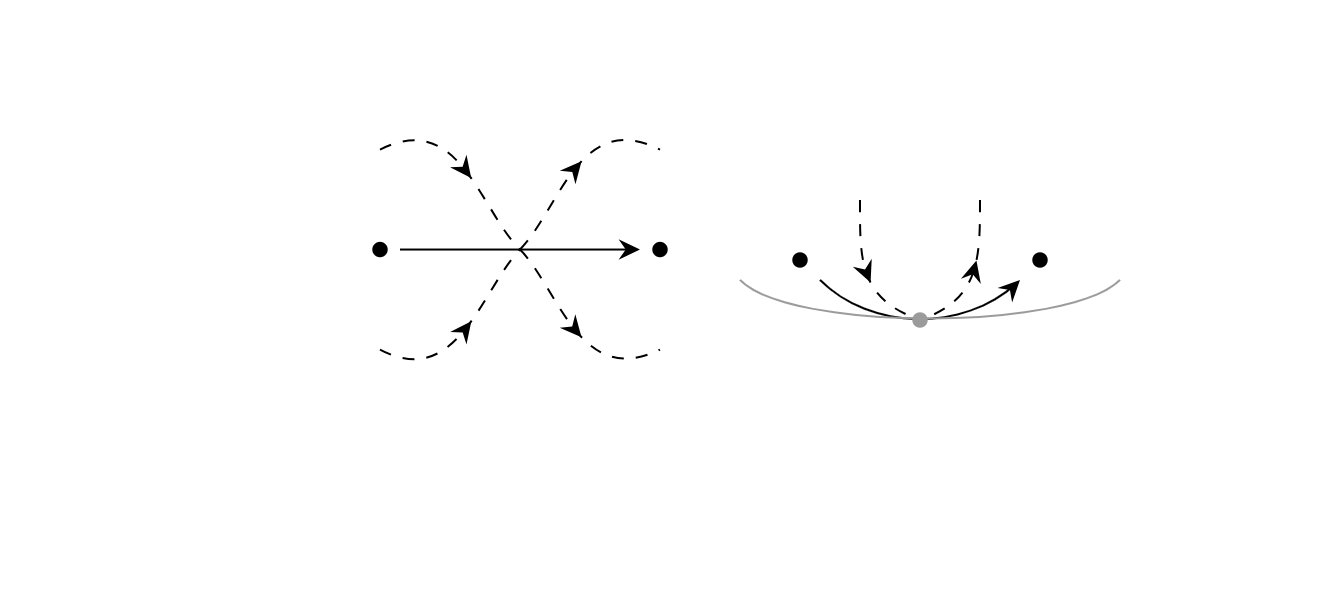}
\end{center}

\begin{Def}\label{def:Q-of-D}
The quiver $Q_D$ associated to a Postnikov diagram $D$ has vertex set $Q_0$ given by the alternating regions. The arrows are given by the intersection points of strands between alternating regions, with orientation as above.
\end{Def}

Notice that we allow arrows between boundary vertices. 

\begin{Ex}
The quiver $Q_D$ obtained from the Postnikov diagram $D$ below recovers the dimer quiver from Figure~\ref{ex:dimer quiver}. The quiver $Q_D$ is drawn on top using solid lines, with the vertices corresponding to the alternating regions in $D$. 
%
\begin{center}
\includegraphics[width=0.95\linewidth]{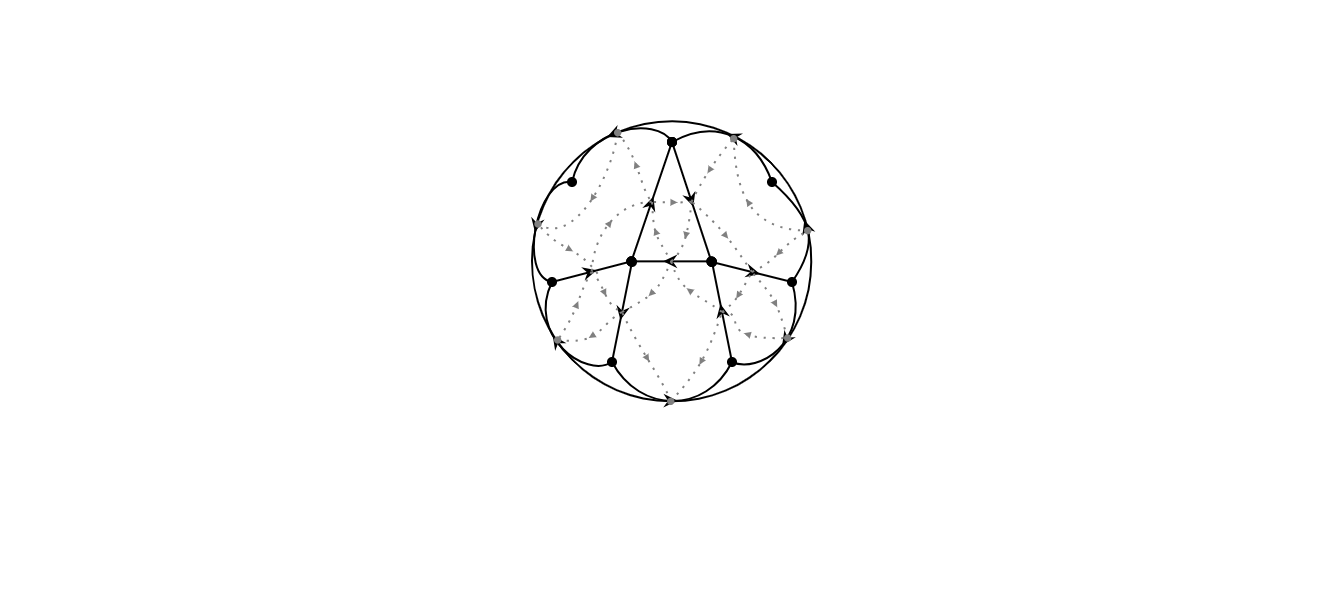}
\end{center}
\end{Ex}



We have seen how a Postnikov diagram gives rise to a quiver $Q_D$. 
In fact, $Q_D$ is a dimer quiver in the disk. See~\cite[Section 2]{BKM} for more details.

The Postnikov diagram $D$ can be recovered from $Q_D$ by taking its strands to be the collection of \textit{zig-zag paths} of the dimer model \cite{GK13} as we 
recall here.


\begin{Def}\label{def:diagram-from-Q}
    Let $Q$ be a dimer quiver in the disk, with $n$ vertices on the boundary. 
    We draw two short crossing oriented segments on every arrow of $Q$, in the orientation of the arrows as on the left and in the middle of Figure~\ref{fig:Q-to-diagram}. Then we connect the segments inside each face to obtain $n$ oriented curves starting and ending at the boundary of the disk, as on the right of the Figure. We write $D_Q$ for the resulting diagram. It is sometimes called the diagram of zig-zag paths in $Q$. 
\end{Def}

In Section~\ref{sec:surfaces}, we will lift this correspondence to 
the more general set-up of dimer quivers on arbitrary surfaces and of so-called ``weak strand diagrams''.

\begin{figure}[H]
    \centering
    \includegraphics[width=12cm]{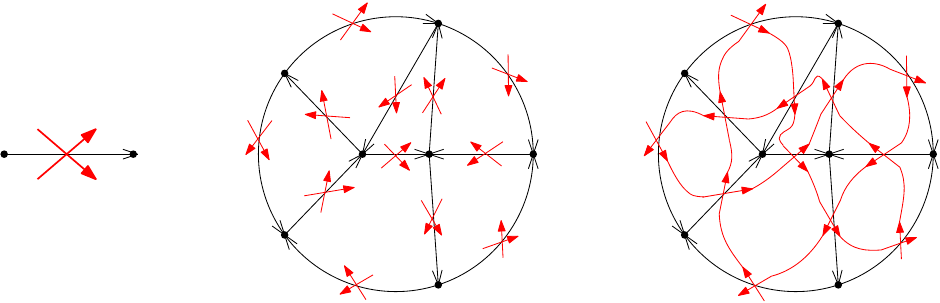}
    \caption{Associating a strand diagram to a dimer quiver}
    \label{fig:Q-to-diagram}
\end{figure}

\begin{Def} \cite[Section 4]{BKM}
 For any arrow $\alpha: I \rightarrow J$ in $Q_1(D)$, let $c$ be the starting label of the strand crossing $\alpha$ from right to left and $d $ the index of the strand crossing $\alpha$ from left to right. In other words, 
 $J=I\setminus \{c\} 
 \cup \{d\}$. 
 The {\em weight of $\alpha$} is then defined to be the cyclic interval from $c$ to $d$ in $\left[n\right]$
$$
w_\alpha=[c, d).
$$
The weight of a path in $Q_D$ is then the sum of the indicator functions on the weights of arrows in the path. A path is called {\em sincere} if the support of its weight is all of $[n]$, and {\em insincere} otherwise. 
\end{Def}

\begin{Def} A dimer algebra $A_Q$ is called {\em thin} if for any pair of vertices $a,b \in Q$ there is a path $h\in e_a A_{Q}e_b$ such that $e_a A_{Q}e_b = \mathbb{C}[[t]]h$.
In other words, there is a path $h:a \to b$ such that for any path $g: a \to b$ in $A_Q$, $h=g t^m$ for some $m \geq 0$. Such a path will be called {\em minimal}.
\end{Def}

\begin{Rem}\label{rem:insincere-is-minimal}
Note that insincere paths are minimal, we will use this in the proof of Lemma~\ref{lem:disk bdy} (2), cf.~\cite[Corollary 9.4]{BKM}.
\end{Rem}

\begin{Prop}\cite[Proposition 2.11]{PresslandCY}\label{Prop:Post-is-Thin}
Let $D$ be a Postnikov diagram in the disk. If the associated quiver $Q_D$ is connected, then $A_{Q_D}$ is thin.
\end{Prop}

\begin{lem}\cite[Lemma 2.13]{PresslandCY}\label{lem:minimal-composition-decomposes-Into-Minimal-parts}
Let $Q$ be a connected dimer quiver such that $A_Q$ is thin. If $g: a \to b$ and $h: b \to c$ are paths in $Q$ such that the composition $gh$ is minimal, then both $g$ and $h$ are minimal.    
\end{lem} 

We now introduce some notation which will be useful when describing paths between boundary vertices. 

\begin{notation}\label{not:hat-and-u}
Let $D$ be a $(k,n)$-diagram in the disk, and let $\alpha$ be a boundary arrow in $Q_D$. As each boundary arrow $\alpha$ in a dimer quiver $Q$ belongs to precisely one face $F_\alpha \in Q_2$, we can uniquely refer to the complement $\widehat{\alpha}$ of $\alpha$ around $F_\alpha$ in the dimer algebra $A_Q$ (ie. $\alpha \widehat{\alpha}=\partial F_\alpha$). 

For each $i\in[n]$ we define paths $u_i:i\to i+1$ and $v_i:i+1\to i$ as follows: 
Let $\alpha$ be the boundary arrow between $i$ and $i+1$. If $\alpha$ is oriented clockwise (i.e. going from $i$ to $i+1$, we set $u_i$ to be this boundary arrow. If $\alpha$ is oriented anticlockwise (i.e. $i+1\to i$), we set $u_i$ to be $\hat\alpha$. Figure ~\ref{fig:minimal paths} gives an illustration of the paths $u,v$ on a dimer quiver.

\begin{figure}[H]
\centering \includegraphics[width=.9\linewidth]{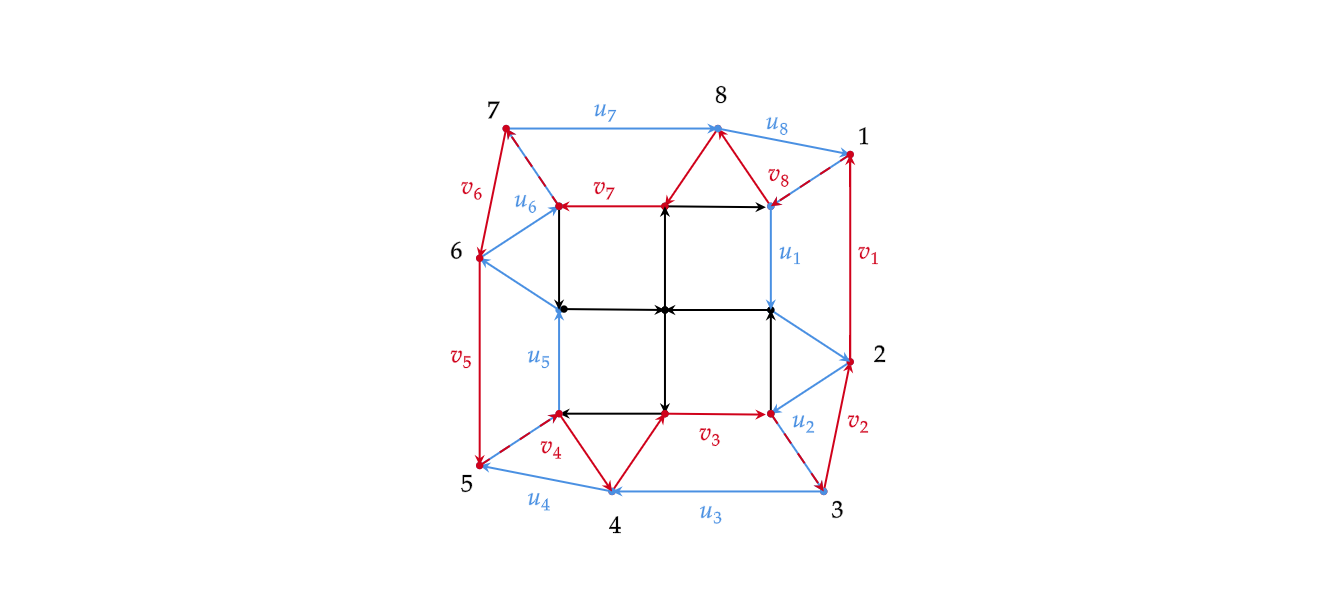} 
\caption{Paths $u$ (in blue) and $v$ (in red).}\label{fig:minimal paths}
\end{figure}

Take $m \in [n]$ and $h \geq 0$. Going forward we will use the following shorthand for composing $h$ consecutive $u_i$'s ($v_i$'s) starting at the vertex $m$ on the boundary of $Q$:
\begin{itemize}
    \item  $u^h_m:= u_{m}u_{m+1}..u_{m+h-1}: m \to m+h$ 
    \item  $v^h_m:= v_{m-1}v_{m-2}\dots v_{m-(h-2)}: m \to m-h$     
\end{itemize}
where addition and subtraction in the indices is done modulo $[n]$. When the starting point is clear we will drop the index $m$ from the notation.
\end{notation}

\begin{lem}\label{lem:disk bdy} 
Let $D$ be a $(k,n)$-diagram in the disk. 
The paths defined above satisfy:
\begin{enumerate}
    \item $v^k = u^{n-k}$ in $A_{Q_D}$,
    \item for any pair of boundary vertices $a,b$ in $Q_D$, a minimal path from $a$ to $b$ in $A_{Q_D}$ is given by $u^h$ or $v^h$ for some $0 \leq h \leq n$. 
\end{enumerate}

\end{lem}

To obtain this result, one can explicitly determine the image of arrows under the inverse of the isomorphism given in ~\cite{BKM}, or by analysis of perfect matching modules [Claim 3.4, ~\cite{CKP}]. However we provide a direct argument by associating weights to the 
arrows in $Q_D$. 

\begin{proof}[Proof of Lemma~\ref{lem:disk bdy}]
The vertices of $Q_D$ arise from $k$-subsets. For $a\in[n]$ we will write $I_a$ 
for the $k$-subset $\{a,a+1,\dots, a+k-1\}$. 
(1) Adjacent boundary vertices in $D$ can be assigned the $k$ element subsets $I_a$ and $I_{a+1}$ of $[n]$ (reducing modulo $n$), for some $a\in[n]$. 
It is clear that the weight of each path $v:I_{a+1}\to I_a$ has the support $[a+k,a)$, and the weight of each path $u:I_a \to I_{a+1}$ has the support $[a,a+k)]$. 
Therefore, the the support of $v^k$ (ending at $e_{I_a}$) will be 
$$\bigcup_{j=0}^{k-1}[a+k+j,a+j).$$ 
The cardinality of this set is $(n-k) + (k-1)=n-1$.
Similarly, the weight of $u^{n-k}$ (starting at $e_{I_a})$ will be
$$\bigcup_{j=0}^{n-k-1}[a+j,a+k+j)$$
with cardinality $k+ (n-k-1)=n-1$. 
Thus we have shown that $v^k$ and $u^{n-k}$ are insincere paths between fixed vertices in $A_Q$. 
Since there is only a unique insincere path between any two vertices by \cite[Corollary 9.4]{BKM}, the claim follows. 

(2) Set $d:= |b-a|$. If $d \in \{k,n-k\}$ then the paths constructed in the proof of (1) are minimal between $I_b$ and $I_a$. 
If $d<k$, then $v^d$ is a subpath of $v^k$ and therefore insincere. By Remark ~\ref{rem:insincere-is-minimal} $v^d$ is minimal. 
If $d>k$ then we have $n-d<n-k$ and similarly $u^{n-d}$ is a subpath of $u^{n-k}$ and therefore insincere and once again minimal.
\end{proof}


\section{Diagrams on surfaces}\label{sec:surfaces}


The notion of a weak Postnikov diagram was introduced in~\cite{BKM} - relaxing the conditions of Definition~\ref{def:pos-diag} and allowing  arbitrary surfaces. Let $\Sigma$ be a oriented surface with marked points on the connected components of the boundary. We assume that every connected component contains at least one marked point. The surface $\Sigma$ may consist of several connected components. 

\begin{Def}\label{def:weak-diagram}
 A \textit{weak Postnikov diagram} on a surface $\Sigma$ consists of oriented curves in $\Sigma$, where each marked point is the source of a unique strand and the target of a unique strand. The diagram must also satisfy the local axioms.
\begin{enumerate}
    \item There are finitely many intersections, and each intersection is  transversal of order two.
     \item Following a fixed strand, the strands that cross it must alternate between crossing from the left and from the right.
\end{enumerate}
\end{Def}
In particular, we allow strands to have self-intersections and we do not ask for condition (4) of Definition~\ref{def:pos-diag} to hold - unoriented lenses may occur. 


Weak Postnikov diagrams are considered up to isotopy that preserves crossings. Notice that a weak Postnikov diagram determines a permutation of the marked points of $\Sigma$. 

Let $C_1,\dots, C_t$ be the connected components of the boundary of $\Sigma$, for some $t>0$. 
We label the marked points of $\Sigma$ clockwise around 
each connected component 
$C_j$ of the boundary of $\Sigma$ by $p_{j,1},\dots, p_{j,n_j}$. 

\begin{Ex}\label{ex:weak-P-d}
Figure~\ref{fig:annulus-diagrams} shows two weak Postnikov diagrams on an annulus with three points on the outer boundary and two points on the inner boundary. 
In cycle notation, 
the permutations these diagrams induce are $(13\overline{1})$ for the one on the left and $(132)$ for the one on the right. Both permutations have two fixed points. 
\begin{figure}
    \centering
\includegraphics[scale=.7]{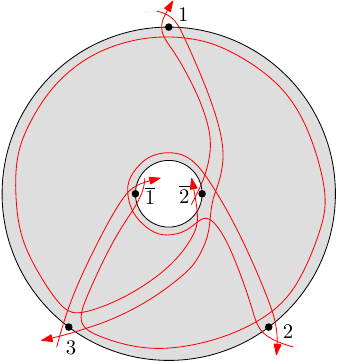}
\hskip 1.5cm
\includegraphics[scale=.7]{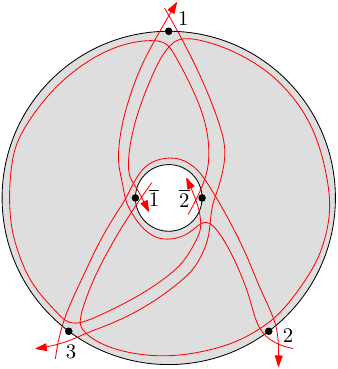}
    \caption{Two  Postnikov diagrams on an annulus.}
    \label{fig:annulus-diagrams}
\end{figure}
\end{Ex}

\begin{Def}\label{def:degree}
A (weak) Postnikov diagram $D$ on a surface $\Sigma$ is said to be of {\em degree $k$} if the permutation induced by $D$ is of the form $p_{j,i}\mapsto p_{j,i+k}$ (reducing modulo $n_j$) 
at every connected component $C_j$ of the boundary of $\Sigma$. 
\end{Def} 

The second picture in 
Figure~\ref{fig:annulus-diagrams} shows a weak Postnikov diagram of degree 2. 

To any weak Postnikov $D$ diagram on a surface, we can associate a quiver $Q_D$ as for the disk case (Definition~\ref{def:Q-of-D}). 
This is also a dimer quiver, following the reasoning 
of~\cite[Remark 3.4]{BKM}, so we have: 

\begin{lem}\label{lm:weak-diag-dimer}
    Let $D$ be a weak Postnikov diagram on a surface $\Sigma$. Then $Q_D$ is a dimer quiver on $\Sigma$. 
\end{lem}


\begin{Ex}
    The dimer quivers of the weak Postnikov diagrams from Example~\ref{ex:weak-P-d} are in Figure~\ref{fig:Q-weak}.
\begin{figure}
\centering
\includegraphics[scale=.7]{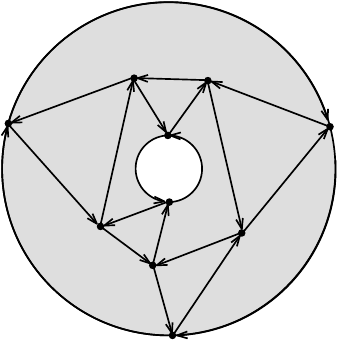}
\hskip 1.5cm 
\includegraphics[scale=.7]{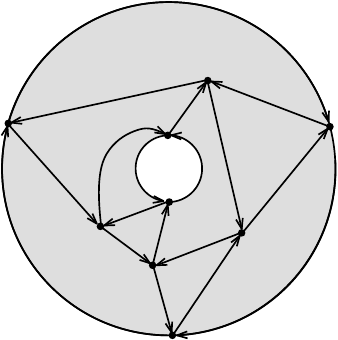}
\caption{Dimer quivers of weak Postnikov diagrams.}
\label{fig:Q-weak}
\end{figure}

\end{Ex}

\begin{Rem}\label{rem:Q-gives-D}
    Recall that to any dimer quiver on the disk, the collection of its zig-zag paths (Definition~\ref{def:diagram-from-Q}) is a Postnikov diagram. The zig-zag paths can be drawn for dimer quivers on arbitrary surfaces as this only relies on the quiver and not on the surfaces. This results in a collection of oriented curves (strands) which will denote by $D_Q$ (as in the disk case, Definition~\ref{def:diagram-from-Q}). We now show that $D_Q$ is a weak Postnikov diagram.
\end{Rem} 

\begin{lem}\label{lem:dimer-gives-weak-diagram}
    Let $Q$ be a dimer quiver on a surface $\Sigma$. Then the collection of alternating strands $D_Q$ is a weak Postnikov diagram on $\Sigma$. 
\end{lem}

\begin{proof}
The collection of zig-zag paths of a dimer model give a strand diagram satisfying the local conditions of a Postnikov diagram (ie. a weak Postnikov diagram) but not necessarily the global conditions, arguing as in~\cite[Remark 3.4]{BKM}. 
\end{proof}

By Lemma~\ref{lm:weak-diag-dimer}, every weak Postnikov diagram $D$ defines a dimer algebra 
$A_Q=A_{Q_D}=\mathbb{C}Q_D/I$ where $I$ is the ideal given by the relations from the internal arrows, cf. Definition~\ref{def:dimer-algebra}. 

Setting $e$ to be the sum of idempotents of $A_Q$ corresponding to boundary vertices in $Q_D$, we thus obtain the boundary algebra of $Q_D$ (Definition~\ref{def:boundary-alg}). 

The boundary algebra associated to a weak Postnikov diagram on a general surface is invariant under twisting/untwisting and geometric exchange: Let $D$ and $D'$ be two weak Postnikov diagrams on $\Sigma$ that can be linked by a sequence of geometric exchanges and (un)twisting moves. Let $Q_D$ and $Q_{D'}$ be the associated dimer quivers. These have the same boundary vertices (i.e. the same $k$-subsets at the alternating regions at the boundary of $\Sigma$). So the associated algebras $A_D$ and $A_{D'}$ have the same boundary idempotents. )

Arguing similarly as in ~\cite[Section 12]{BKM} (in particular, Lemma 12.1 and Corollary 12.4), now in the set-up of weak Postnikov diagrams, one gets the following result: 

\begin{lem} \label{lm:invariance}
     Let $D, D^{\prime}$ be weak Postnikov diagrams on $\Sigma$, let $e$ and $e'$ be the sum of the idempotents corresponding to boundary vertices of $Q_D$ and $e'$ be the sum of the idempotents corresponding to the boundary vertices of $Q_{D'}$. Assume that $D'$ can be obtained from $D$ by a sequence of geometric exchanges and twisting or untwisting moves. Then the corresponding boundary algebras $e A_D e$ and $e^{\prime} A_{D^{\prime}} e^{\prime}$ are isomorphic.
\end{lem}








\section{Dimer Gluing}\label{sec:MV}




From now on, we consider dimer quivers with several connected components. If $Q$ is a dimer quiver which is not connected, then the associated surface $\Sigma=\Sigma_Q$ (see Remark~\ref{rem:dimer-surface}) 
has several connected components. 

We will choose pairs of connected sets of arrows along the boundary of the dimer quiver and glue $Q$ to itself by identifying these arrows pairwise. 

This gluing operation will be done in such a way that important properties of the dimer quivers are preserved.


\begin{Def}\label{def:arrows-to-glue}
    Let $Q$ be a dimer quiver and $\Sigma=\Sigma_Q$ the underlying surface. 
Let $I$ and $J$ be two disjoint connected sets of $s>0$ boundary arrows of $Q$, 
let $\{v_0,\dots, v_{s}\}$ and 
$\{w_0,\dots, w_{s}\}$ be 
the vertices of $I$ and of $J$, respectively, where $v_m\ne w_m$ for $m=1,\dots, s-1$. 
We label the vertices and arrows of $I$ clockwise and the vertices and arrows of $J$ anticlockwise along the boundary of $Q$: for every $m\in\{1,\dots, s\}$, 
$i_m$ is an arrow $v_{m-1}\to v_m$ or $v_m\to v_{m-1}$ and 
$j_m$ is an arrow $w_{m-1}\to w_m$ or $w_m\to w_{m-1}$. 
The sets $I$ and $J$ may belong to two different connected components of the boundary of $Q$. 
See Figures~\ref{fig:Q-and-rhoQ-disk} and~\ref{fig:Q-and-rhoQ} for two examples. 

We say that the two arrows $i_m$ and 
$j_m$ are {\em parallel}, 
if $i_m:v_m\to v_{m-1}$ and $j_m:w_m\to w_{m-1}$ or if $i_m:v_{m-1}\to v_{m}$ and $j_m:w_{m-1}\to w_{m}$. 
\end{Def}

We now define a map 
$\rho=\rho_{I,J}$ on dimer quivers with chosen sets 
$I$, $J$ of arrows along the boundary. This map adds new boundary arrows between the vertices of $J$ whenever the arrows of $I$ and $J$ are not parallel, keeping the set of vertices fixed. For every added arrow, the resulting quiver $\rho(Q)$ has a two-cycle. 

\begin{Def}[The quiver $\rho(Q)=\rho_{I,J}(Q)$]\label{def:rhoQ} 
Let $I$ and $J$ be as in Definition~\ref{def:arrows-to-glue}. 
Let $m$ be in $\{1,\dots, s\}$. 
If the arrows $i_m$ and $j_m$ are parallel, we set $\rho(j_m):=\rho_{I,J}(j_m):=j_m$. 
If $i_m$ and $j_m$ are not parallel, 
we add a new 
arrow $\rho(j_m):=\rho_{I,J}(j_m)$ in the opposite direction to $j_m$, thereby creating an oriented digon between the two endpoints $w_{m-1}$ and $w_m$ of $j_m$. 

Let $\rho(J)=\rho_{I,J}=\{\rho(j_1),\dots,\rho(j_s)\}$ and let $\rho(Q)$ be the quiver obtained from adding an 
arrow $\rho(j_m)$ for every such occurrence of non-parallel arrows. 
\end{Def}

\begin{figure}
    \centering    \includegraphics[width=10cm]{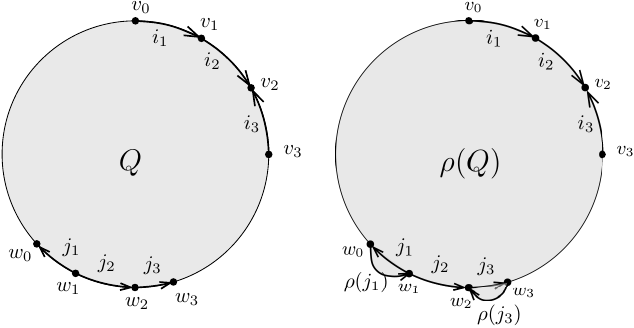}
    \caption{The quivers $Q$ and $\rho(Q)$: adding boundary arrows.}
    \label{fig:Q-and-rhoQ-disk}
\end{figure}

\begin{figure}
    \centering
    \includegraphics[width=10cm]{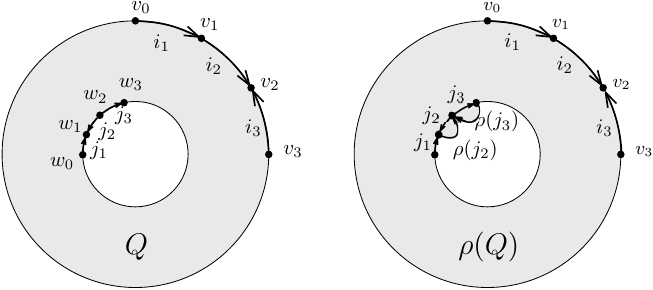}  
    \caption{The quivers $Q$ and $\rho(Q)$}
    \label{fig:Q-and-rhoQ}
\end{figure}

\begin{lem}\label{lm:rhoQ-dimer}
Let $Q$ be a dimer quiver. Then $\rho(Q)$ is a dimer quiver. 
\end{lem}

\begin{proof} 
 The new quiver has no loops. For condition (2) of Definition~\ref{def:dimer-quiver}, we only have to check the arrows $j$ with $j\ne \rho(j)$: Every such $j$ is an internal arrow of $\rho(Q)$ and the two faces it is incident with have opposite sign by construction, the new arrows $\rho(j)$ are boundary arrows. 
Adding the arrows $\rho(j)$ increases the connectivity of the incidence graph and so condition (3) of Definition~\ref{def:dimer-quiver} is also satisfied. 
\end{proof}

Note that the surface $\Sigma(\rho(Q))$ is the same as $\Sigma(Q)$, up to homotopy. 

\vskip .2cm

In Section~\ref{sec:bridge}, we will use the quiver $\rho(Q)$ to create dimer models on annuli. 

\begin{lem}\label{lm:rhoQ has same dimer alg} Let $Q$ and $\rho(Q)$ be as in 
Definition~\ref{def:rhoQ}. Then $A_Q \cong A_{\rho(Q)}.$    
\end{lem}
\begin{proof}
    Notice that the canonical inclusion of $Q$ into $\rho(Q)$ induces a map $\mathbb{C}Q \hookrightarrow \mathbb{C}(\rho(Q))$ that descends to a homomorphism $\psi: A_Q \to A_{\rho(Q)}$. Let $J^\prime\subseteq J$ be the set of arrows in $\rho(Q)$ for which $\rho(j_m)\neq j_m$ and let $K:=\{\rho(j): j \in J^\prime\}$ be the corresponding subset of boundary arrows of $\rho(Q)$. 
     
    Each $\alpha \in J^\prime$ is an internal arrow of $\rho(Q)$, and so belongs to the boundary of faces $F_1,F_2 \in \rho(Q)_2$ where $\partial F_1 = j \rho(j)$ and $\partial F_2 = j p_{F_2}$ where $p_{F_2}$ is the complement of $j$ in $\partial F_2$. The corresponding relation $\partial_\alpha(W)$ implies $\rho(j)= p_{F_2}$.

    Define $\varphi:\mathbb{C}{\rho(Q)} \to \mathbb{C}Q$ to be the identity on vertices in $\rho(Q)_0$ (viewed as vertices of $Q$) and the identity on arrows in $ \rho(Q)_1 \backslash K$ (viewed as arrows of $\rho(Q)$). For each $\alpha \in K$ define $\varphi(\alpha):= p_{F_2}$. Then $\varphi$ descends to a homomorphism $A_{\rho(Q)} \to A_Q$ that is the inverse of $\psi$.
\end{proof}

\begin{Def}[The quiver $Q_{I\equiv J}$ obtained from gluing]
\label{def:gluing}
Let $Q$ be a dimer quiver. 
Let $I=\{i_1,\dots, i_s\}$ and $J=\{j_1,\dots, j_s\}$ be disjoint 
connected sets of boundary arrows of $Q$ and let $\rho(i_m)$, $1\le m\le s$, be as in Definition~\ref{def:rhoQ}. 

The quiver $Q_{I\equiv J}$ obtained from $\rho(Q)$ by 
identifying $v_m$ with $v'_m$ for 
$m=0,\dots, s+1$ and 
identifying $i_m$ with $\rho(j_m)$ for every $1\le m\le s$ is called the {\em quiver arising from $Q$ by gluing $I$ with $J$}. \end{Def}



\begin{lem}\label{lm:gluing-makes-dimer}
    Let $Q$ be a dimer quiver, let $I$ and $J$ be as in Definition~\ref{def:gluing}. Then the quiver $Q_{I\equiv J}$
    is a dimer quiver. 
\end{lem}
\begin{proof}
    By Lemma~\ref{lm:rhoQ-dimer}, $\rho(Q)$ is a dimer quiver. Condition (3) of Definition~\ref{def:dimer-quiver} is satisfied as the effect of gluing is to increase the connectivity of the incidence graphs of the arrows in $I$ or $J$.  (1) is satisfied as gluing does not introduce any loops as no arrows are contracted. For (2): any arrow of $I$ and of $\rho(J)$ is an internal arrow of $\rho(Q)$, between two faces of opposite orientations. 
\end{proof} 

\begin{Rem}\label{rem:properties-gluing}
(1) The resulting quiver $Q_{I\equiv J}$ lives on the surface 
obtained by gluing the parts of the boundary of $\Sigma$ along the identified arrows. 

(2) 
The quiver $Q_{I\equiv J}$ is not necessarily reduced even if $Q$ is: for every 
arrow $j\in J$ with 
$j\ne \rho(j)$, the construction introduces an internal two-cycle. 
\end{Rem}

We can iterate the construction from Definition~\ref{def:gluing} and glue along 
several intervals of arrows: 

\begin{Def}
    Let $Q$ be a dimer quiver. 
    Let $I_1,\dots, I_m$ and $J_1,\dots, J_m$ be disjoint connected sets of boundary arrows of $Q$ (as in Definition~\ref{def:arrows-to-glue}). 
    For $t\in \{2,\dots, m\}$ set 
    \[
    Q_{I_1\equiv J_1,\dots, I_t\equiv J_t} := (Q_{I_1\equiv J_1,\dots, I_{t-1}\equiv J_{t-1}})_{I_t\equiv J_t}
    \]
\end{Def}

\begin{Rem}\label{rem:gluing-order}
One can check that iterated gluing is independent of the order: 
\[
Q_{I_\sigma(1)\equiv J_\sigma(1),\dots, I_\sigma(t)\equiv J_\sigma(t)} = 
Q_{I_1\equiv J_1,\dots, I_m\equiv J_m}
\]
for any permutation $\sigma$ of $\{1,\dots, m\}$. 
\end{Rem}

In Section~\ref{sec:annulus}, we will consider dimer quivers consisting of two disks and glue them along two sets of arrows to construct dimer quivers on annuli.


\begin{Ex}\label{ex:first-gluings}
Let $Q$ be the dimer quiver with two connected components as in Figure~\ref{fig:A-B}. 

(1) Let $I=\{\alpha_4,\alpha_3\}$ and 
$J=\{\beta_1,\beta_2\}$. 
Then $\rho(J)=J$ and $Q_{I\equiv J}$ is a dimer quiver on a disk, as on the left of Figure~\ref{fig:glued}.  

(2) Let $I_1=\{\alpha_2\}$, 
$I_2=\{\alpha_5\}$, $J_1=\{\beta_5\}$ and 
$J_2=\{\beta_2\}$. Then 
$\rho(J_1)=\{\rho(\beta_5)\}$ and 
$\rho(J_2)=J_2$ and 
$Q_{I_1\equiv J_1,I_2\equiv J_2}$ is a dimer quiver on an annulus, as on the right of Figure~\ref{fig:glued}. 
\end{Ex}

\begin{figure}
    \centering
    \includegraphics[scale=.9]{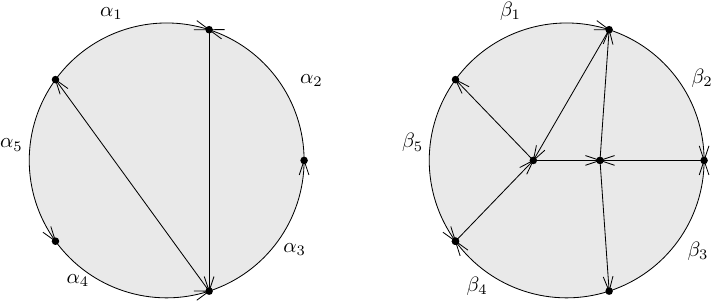}
    \caption{A dimer quiver with two 
    connected components.} 
    \label{fig:A-B}
\end{figure}

\begin{figure}
    \centering
    \includegraphics[scale=.9]{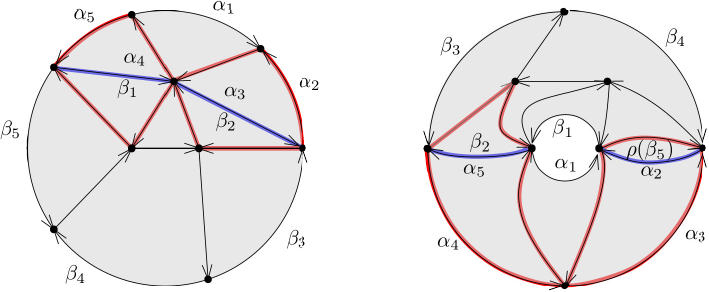}
    \caption{Two dimer quivers obtained from the quiver of Figure~\ref{fig:A-B} by gluing. The red and blue indicate the seams - the red paths correspond to complements of arrows of the sets of glued arrows. }
    \label{fig:glued}
\end{figure}

Recall that if $\beta$ is a boundary arrow of a dimer quiver, we write $\widehat{\beta}$ for the complement of $\beta$ around the unique face $\beta$ belongs to (Notation~\ref{not:hat-and-u}). 

\begin{Def}\label{def:seam}
  Let $r\in \{1,\dots,m\}$ and $s_r:=|I_r|=|J_r|$. 
  We refer to the  subquiver $I_r=J_r$ of $Q_{I_1\equiv J_1,\dots, I_m\equiv J_m}$ on the paths $\{i_{r_m},\rho(j_{r_m}), \widehat{i_{r_m}},\widehat{\rho(j_{r_{m}})}: 1 \leq m\leq s_r  \}$ as the {\em seam at $I_r=J_r$. }
\end{Def}

To illustrate the notion of a seam: the disk in Figure~\ref{fig:glued} has a single seam given by the paths $\{\alpha_4,\alpha_3,\beta_1,\beta_2, \widehat{\alpha_4},\widehat{\alpha_3},\widehat{\beta_1},\widehat{\beta_2}\}$ The annulus in ~\ref{fig:glued} has a seam at $\alpha_5 \equiv \beta_2$ given by the paths $\{\alpha_5,\beta_2, \widehat{\alpha_5}, \widehat{\beta_2}\}$ and a seam at $\alpha_2 \equiv \beta_5$ given by  $\{\alpha_2,\beta_5, \widehat{\alpha_2}, \widehat{\beta_5}\}$.

\begin{Def}\label{def:seam_crossing}
Let $g$ be a path in $Q_{I\equiv J}$. 
We say that $g$ crosses the seam $I\equiv J$ if $g$ contains a subpath $\sigma = g_0g_1g_2$ where $g_1:v_0\to v_2$ is a path of arrows completely contained in $I\equiv J$ and $v_0, v_2$ are the only vertices of $g_0$ and $g_2$ respectively belonging to $I \equiv J$, where both $g_0$ and $g_2$ are not trivial paths. 
\end{Def}

We fix the following notation for the remainder of the section. Let $Q$ be a dimer quiver, and $I=\{i_1,\dots, i_s\}$ and $J=\{j_1,\dots, j_s\}$ be disjoint connected sets of boundary arrows of $Q$ as in 
Definition~\ref{def:arrows-to-glue}. Let $h, t:Q_1 \to Q_0$ denote the source and target maps of the quiver $Q$. Let 
$\mathcal{R}=\mathcal R(I,J)$ denote the following set of paths in $\mathbb{C}(\rho(Q))$: 
\begin{equation}\label{eq:curlyR}
\mathcal{R}:=\{ i_m -\rho(j_m), \widehat{i_m}-\widehat{\rho(j_m)}, e_{h(i_m)}-e_{h(\rho(j_m))} : 1 \leq m \leq s \} \cup \{ e_{t(i_s)}- e_{t(\rho(j_s))}\}.    
\end{equation}

In this setting, the glued dimer algebra can be determined from the initial dimer algebra in the following sense.
\begin{Prop}\label{prop:algebra-glued} 
The dimer algebra $A_{Q_{I\equiv J}}$ is isomorphic to $A_{\rho(Q)}/\langle\mathcal R\rangle.$   
\end{Prop}

\begin{proof} 
Let $\mathcal{T}$ denote the set of paths  $$\mathcal{T}:=\{ i_m -\rho(j_m), e_{h(i_m)}-e_{h(\rho(j_m))} : 1 \leq m \leq s \} \cup \{ e_{t(i_s)}- e_{t(\rho(j_s))}\}$$
in $\mathbb{C}(\rho(Q))$. Then $\langle \mathcal{R} \rangle = \langle \mathcal{T} \rangle + \langle \widehat{i_m}-\widehat{\rho(j_m)} : 1 \leq m \leq s \rangle$. Observe that we have an isomorphism $\psi: \mathbb{C}Q_{I \equiv J} \cong \mathbb{C}(\rho(Q))/\langle \mathcal{T}\rangle$ that is the identity on arrows and on the trivial paths at vertices.

Let $V$ be the ideal of relations of $\mathbb{C}Q_{I\equiv J}$ defined by the internal arrows of $Q_{I\equiv J}$ (Definition~\ref{def:relations-potential}), let $U$ be the ideal of relations of $\mathbb{C}(\rho(Q))$ defined by the internal arrows of $\rho(Q)$, so 
\[
A_{Q_{I\equiv J}}=\mathbb{C}Q_{I\equiv J}/V  
\quad \mbox{and}
\quad 
A_{\rho(Q)}=\mathbb{C}(\rho(Q))/U
\].
Since every internal arrow of $\rho(Q)$ is an internal arrow of $Q_{I\equiv J}$, the set of relations generating $U$ is contained in $\psi(V)$. 
The additional relations for $\psi(V)$ arise from boundary arrows of $\rho(Q)$ which are internal in the glued quiver $Q_{I\equiv J}$. These are given by the set 
$\langle i_m-\rho(j_m), e_{h(i_m)}-e_{h(\rho(j_m))}: 1\le m\le s\rangle$. 
These are exactly the relations which are in $\mathcal R$ but not in $\mathcal T$. 
We therefore get an 
isomorphism 
\[
(\mathcal{C}(\rho(Q))/\langle\mathcal T \rangle )/V \cong 
(\mathcal{C}(\rho(Q))/U)/\mathcal R.
\]
Putting everything together yields 
\begin{align*}
   A_{Q_{I \equiv J}} & = \mathbb{C}Q_{I \equiv J} / V \\
      & \cong \left(\mathbb{C}{\rho(Q)}/\mathcal{T}\right)/V \\
      & \cong \left(\mathbb{C}{\rho(Q)}/U\right)/\mathcal{R} \\
      & =  A_{\rho(Q)}/ \mathcal{R}
\end{align*}
as claimed.       

    \end{proof}


\begin{Rem}\label{rem:glued-bdy-alg}
    Note that the boundary algebra of the dimer quiver $Q_{I \equiv J}$ is by definition equal to $e'Q_{I \equiv J}e'$ where $e'$ is the sum of the 
    idempotent elements of $A_{Q_{I \equiv J}}$ corresponding to the boundary vertices of the quiver $Q_{I \equiv J}$. 
    In terms of the boundary of $Q$, the element $e'$ is the sum $\sum_{i \in X} e_i$ where $X$ is the union of the boundary vertices of $Q$ excluding those which are both a source and target of an arrow of $I$ and of $J$, as these vertices are internal vertices of the surface obtained from gluing. 
\end{Rem}

The following result explains how the boundary algebra of a glued dimer quiver can be obtained from the two boundary algebras of the individual quivers when the individual dimer algebras as well as the glued dimer algebra are thin. 
In the statement of the following two results, we will write $e_{\alpha}$ (respectively $e_{\beta}$) for the sum of all idempotents corresponding to the boundary vertices of the quiver $Q_{\alpha}$ (respectively of $Q_{\beta}$). 
As before, fix $I$ and $J$ two disjoint sets of boundary arrows with $|I|=|J|$ and with relations $\mathcal{R}$ as in (\ref{eq:curlyR}). 
Let $e'$ be the sum of idempotent elements corresponding to the boundary vertices of $Q_{I\equiv J}$.

\begin{Th} \label{th:bdy-alg} Let $Q,I,J, \mathcal{R}$ and $e'$ be as above. Assume that $Q$ can be  expressed as a disjoint union $Q= Q_{\alpha} \amalg Q_{\beta}$ of connected dimer quivers $ Q_{\alpha}, Q_{\beta}$ whose  dimer algebras are thin.  
Let $e_{\alpha}A_{Q_{\alpha}}e_{\alpha} = 
\left<\alpha_1,..,\alpha_{s_1}\right> / \langle \mathcal{R}_1 \rangle$ and $e_{\beta} A_{Q_{\beta}}e_{\beta} = 
\left<\beta_1,..,\beta_{s_2}\right> / \langle \mathcal{R}_2 \rangle$. 
If $A_{Q_{I \equiv J}}$ is thin then we have an isomorphism: 
\[ 
e^\prime A_{Q_{I \equiv J}}e^\prime\cong e^\prime \left(\langle\{\alpha_i\}_i,\{\beta_j\}_j \rangle/\langle \mathcal{R}_1, \mathcal{R}_2, \mathcal{R} \rangle \right)e^\prime
\]  
\end{Th}
In order words, the boundary algebra of $Q_{I\equiv J}$ is determined by those of $A_{Q_{\alpha}}$ and $A_{Q_{\beta}}$. 
\begin{proof}
Viewing $\langle\{\alpha_i\}_i,\{\beta\}_j\rangle$ as a subalgebra of $\mathbb{C}Q$ we inherit the homomorphism 
\[
\phi:e^\prime \left(\langle \{\alpha_i\}_i, \{\beta_j\}_j \rangle \right)e^\prime \to  e^\prime A_{Q_{I \equiv J}}e^\prime. 
\]
given by the composition of homomorphisms
\[
\varphi:
e'\langle\{\alpha_i\}_i,\{\beta\}_j\rangle e'\hookrightarrow 
e'\mathbb{C}Qe' \twoheadrightarrow e'\mathbb{C}Q_{I\equiv J} \twoheadrightarrow e'A_{Q_{I\equiv J}}e'.
\]
We first show that $\phi$ is surjective: As $A_{Q_{I \equiv J}}$ is thin, for any pair of boundary vertices $a,b$ in $A_{Q_{I \equiv J}}$ we have a unique minimal path $g:a \to b$. It suffices to show that each such minimal path $g$ is in the image of $\phi$. If $g$ can be viewed as a path contained entirely in $Q_{\alpha}$ or $ Q_{\beta}$ then $g \in \left<\alpha_1,..,\alpha_s\right>$ or $g \in \left<\beta_1,..,\beta_t\right>$ respectively. In either case, we find $g'\in e'\langle\{\alpha_i\}_i,\{\beta\}_j\rangle e'$ with $\phi(g')=g$. 

Next consider the case where $g:a\to b$ crosses the seam $I\equiv J$; That is $g=g_0g_1$ where $g_0$ is a path completely contained in $Q_{\alpha}$ or $ Q_{\beta}$, without loss of generality assume $g_0$ is a path in $Q_{\alpha}$. By Lemma ~\ref{lem:minimal-composition-decomposes-Into-Minimal-parts} $g_0$ is a minimal between boundary vertices of $Q_{\alpha}$ and can be expressed in terms of $g \in \left<\alpha_1,..,\alpha_s\right>$. By iterating this procedure for each crossing of the seam we obtain an element $g'\in e'\langle\{\alpha_i\}_i,\{\beta\}_j\rangle e'$ with $\phi(g')=g$. 

Concerning the kernel of $\phi$. By Proposition~\ref{prop:algebra-glued}, the dimer algebra of $Q_{I\equiv J}$ can be written as 
$\mathbb{C}(\rho(Q)) /\langle\{
\mathcal{R}_1,\mathcal{R}_2,\mathcal R\} \rangle$, yielding the claim. 
\end{proof}

A useful application of Theorem ~\ref{th:bdy-alg} is that the boundary algebra of any Postnikov diagram in the disk can be determined by recursively splitting it into smaller Postnikov diagrams.

\begin{Cor}\label{cor:splitting} Let $D$ be any Postnikov diagram on the disk with dimer quiver $Q_D$. If 
    $Q_D= Q_{I\equiv J}$ where $Q=Q_{\alpha} \amalg Q_{\beta}$ is a decomposition of  into disjoint, connected dimer quivers, then
    \begin{enumerate}
        \item The strand diagrams $D_\alpha$ and $D_\beta$ corresponding to $Q_{\alpha} $ and $ Q_{\beta}$ are Postnikov diagrams in the disk, and
        \item the boundary algebra of $Q_D$ may be explicitly determined from the boundary algebras of $Q_{\alpha}$ and $ Q_{\beta}$.
    \end{enumerate}
   
\end{Cor}

\begin{proof} 
    (1) As $Q_{\alpha} $ and $ Q_{\beta}$ are dimer quivers, the associated strand diagrams $D_\alpha$ and $D_\beta$ will satisfy the local properties of a Postnikov diagram. To see they also satisfy the global properties, any self-intersection or oriented lens in $D_\alpha$ or $D_\beta$ would necessarily appear in $D$ as $Q_{\alpha} $ and $ Q_{\beta}$ are connected subquivers of $Q$.
    
    (2) By part (1) and Proposition ~\ref{Prop:Post-is-Thin} the dimer algebras $A_{Q_\alpha}$ and $A_{Q_\beta}$ are thin and so Theorem ~\ref{th:bdy-alg} can be applied. 
\end{proof}

\begin{Rem}
    In Section ~\ref{sec:annulus} we illustrate a more general procedure for determining the boundary algebra of $A_{{(Q_{\alpha}\amalg Q_{\beta})}_{I_1\equiv J_1,I_2 \equiv J_2,...,I_m \equiv J_m}}$ from presentations of $e_{\alpha}A_{Q_{\alpha}}e_{\alpha}$ and $e_{\beta} A_{Q_{\beta}}e_{\beta}$ under the additional assumption that  $A_{{(Q_{\alpha}\amalg Q_{\beta})}_{I_q\equiv J_q}}$ is thin for any $1 \leq q \leq m$. 
\end{Rem}





\section{Bridge Quivers}\label{sec:bridge}


In \cite[Section 13]{BKM}, weak Postnikov diagrams arising from surface triangulations were studied. These were shown to be examples of degree $2$ diagrams.

In particular, \cite[Proposition 13.5]{BKM} gives an explicit presentation of the boundary algebra of a weak Postnikov diagram arising from a triangulation of the annulus. 

Our goal is to produce degree $k$ diagrams on surfaces. In \cite[Section 4]{scott}, Scott gives a Postnikov diagram of degree $k$ for any $k,n$, the so-called rectangular arrangements. At present, a general construction of (weak) Postnikov diagrams of degree $k$ is not known, apart from $k=2$ as mentioned above. 

In this section, we give a construction of dimer quivers on annuli whose associated (weak) Postnikov diagrams are of degree $k$. 
We will do so by gluing certain dimer quivers along a collection of arrows, using dimer quivers of degree $k$ together with the so-called {\em bridge} quivers we introduce below. 

\begin{Def}
    Let $M^{(d)}=(m^{(d)}_{ij})$ be the $d^2 \times d^2$ matrix defined inductively for $d \geq 3$ as follows:
    
    \[ M^{(3)} =
    \begin{pmatrix}
        0 & 1 & 0 & 0 & 0 & 0 & 0 & 0 & 0 \\
        0 & 0 & 0 & 0 & 1 & 0 & 0 & 0 & 0 \\
        0 & 0 & 0 & 0 & 0 & 0 & 0 & 0 & 0 \\
        0 & 0 & 0 & 0 & 0 & 0 & 0 & 0 & 0 \\
        1 & 0 & 0 & 0 & 0 & 0 & 0 & 0 & 1 \\
        0 & 0 & 0 & 0 & 0 & 0 & 0 & 0 & 0 \\
        0 & 0 & 0 & 0 & 0 & 0 & 0 & 0 & 0 \\
        0 & 0 & 0 & 0 & 1 & 0 & 0 & 0 & 0 \\
        0 & 0 & 0 & 0 & 0 & 0 & 0 & 1 & 0 
    \end{pmatrix}
   \]
The matrix $M^{(d)}$ is then defined as follows:

For $1 \leq i,j \leq (d+1)(d-2)$ 

\[m^{(d)}_{ij} = \begin{cases} 
      m^{(d-1)}_{i-\lceil i/d \rceil+1,j-\lceil j/d \rceil+1} & i, j \text{ mod } d  \leq d-2 \\
     1 & (i,j)=(td-2,td-1), 1 \leq t \leq d-2 \\
     1 & (i,j)=(td-1,(t+1)d-1), 1 \leq t < d-2 \\
     1 & (i,j)=(td-2,(t+1)d-2), 1 \leq t < d-2 \\
     1 & (i,j)=((t+1)d-1,td-2), 1 \leq t < d-2 \\
      0 & \text{otherwise} 
   \end{cases}
\]

and $ i \text{ or } j > (d+1)(d-2)$ by
\[m^{(d)}_{ij} = \begin{cases} 
     0 & i \text{ or } j \leq (d-2)(d-3) \\
     m^{(d)}_{d^2+1-i,d^2+1-j} & \text{otherwise} 
   \end{cases}
\]
\end{Def}


\begin{Def}\label{def:bridge}
Let $k\ge 2$. 
The {\em $k$-bridge} $\Theta_k$ is the quiver with adjacency matrix $M_{k+1}$. That is, the vertices corresponding to the nonzero rows (equivalently nonzero columns) in $M_{k+1}$ labelled by the $(k+1)^2$ lattice points $[k+1]\times [k+1]=\{(i,j): 1 \leq i,j \leq k+1\}$. There is an arrow from $(i_1,j_1) \to (i_2,j_2)$ if $M^{(k+1)}_{(i_1-1)(k+1)+j_1,(i_2-1)(k+1)+j_2} =1$.
\end{Def}


\begin{Rem}\label{rem:bridge-shape}
Using the lattice points to draw the $k$-bridge, for $k\ge 3$, the quiver $\Theta_k$ is formed by two triangular regions with a right angle at positions $(1,k)$ and $(k+1,2)$ and $k-2$ arrows horizontally, vertically and diagonally (for the hypotenuse) forming their sides (for $k=2$, we only have a single vertex). 
One of these right angled triangles has the vertices $(k+1,k)$, $(k+1,2)$, $(3,2)$. 
The other triangle has vertices $(1,k)$, $(k-1,k)$ and 
$(1,2)$. 
has vertices with sides from $(k+1,k-1)$. 
This is completed by a triangle $(1,1)\to (1,2)\to (2,2)\to (1,1)$, a sequence of $k-2$ squares going diagonally between the vertices $(2,2)$ and $(k,k)$ and a triangle on the vertices $(k+1,k+1)\to(k+1,k)\to (k,k)\to(k+1,k+1)$. See Example~\ref{ex:bridge-small} for $k\le 4$, Figures~\ref{fig:theta-arrows} and Figure~\ref{fig:bridge-shape} for an illustration. 
\end{Rem}

\begin{Ex}\label{ex:bridge-small} 
The $k$-bridges for $2\leq k \leq 4$ are given below.  
\begin{center}  
\includegraphics[width=1\linewidth]{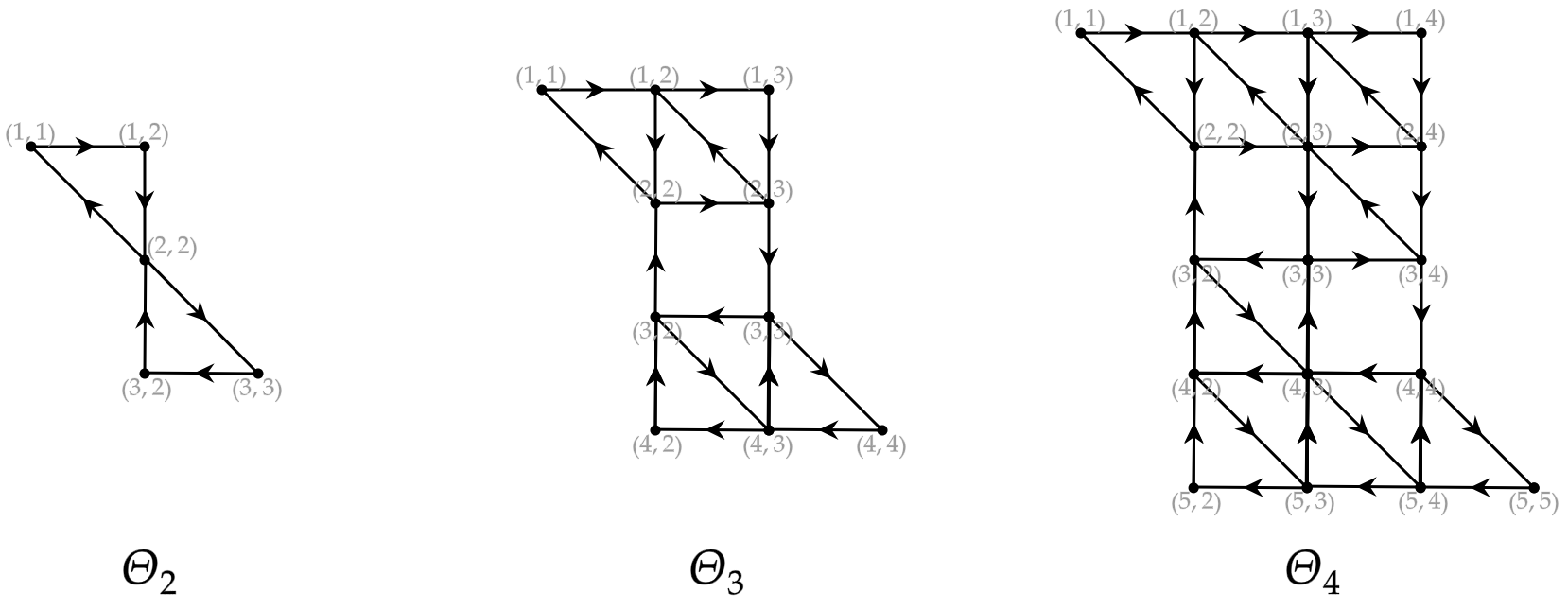} 
\end{center}
\end{Ex}

\begin{Rem} 
For $k>1$ the quiver $\Theta_k$ is a dimer quiver. For $k>2$, the associated surface $\Sigma_{\Theta_k}$ is a disk, for $k=2$ it is two disks meeting at a single vertex. For any $k$, $\Theta_k$ has a rotational symmetry about the centre of the quiver.
\end{Rem}




Recall that  $D_{\Theta_k}$ is the collection of strands arising from the zig-zag paths in $\Theta_k$,  Definition~\ref{def:diagram-from-Q}. 

\begin{Rem}\label{rem:bridge-strands}
One can use the explicit embedding of $\Theta_k$ in the plane to check that the diagram $D_{\Theta_k}$ is homotopic to a collection of strands where each strand is either a straight line between endpoints, or is of an $L$-shape, containing a single right-angle bend to the left. In particular, any two strands starting from the eastern boundary of $D_{\Theta_k}$ (ie. starting at arrows $s_i$, $s_j$ in Figure~\ref{fig:theta-arrows}) will never cross. Symmetrically, any two strands starting from the western boundary of $D_{\Theta_k}$ (ie. starting at arrows $r_i$, $r_j$ in Figure~\ref{fig:theta-arrows}) will also never cross.
\end{Rem}

\begin{lem}\label{lem:theta-is-Post}
    For $k>2$, let $D_{\Theta_k}$ be the strand diagram associated to $\Theta_k$ as in Definition~\ref{def:diagram-from-Q}. Then  $D_{\Theta_k}$ is a Postnikov diagram on the disk.
\end{lem}

\begin{proof}
    By Lemma ~\ref{lem:dimer-gives-weak-diagram} $D_{\Theta_k}$ is a weak Postnikov diagram on the disk, so it remains to show the global axioms ((3) and (4) of Definition ~\ref{def:pos-diag}) hold. From Remark ~\ref{rem:bridge-strands} we can see that (i) there are no self-intersections of strands and (ii) the lenses formed by double-intersections of the $L$-shaped strands are all oriented, and thus (3) and (4) are satisfied. 
\end{proof}

\begin{lem}\label{lem:theta-is-thin}
For $k \geq 2$, the dimer algebra $A_{\Theta_k}$ associated to $\Theta_k$ is thin.    
\end{lem}

\begin{proof}
    When $k=2$, $A_{\Theta_k} \cong \mathbb{C}\Theta_2$ as $\Theta_k$ has no internal arrows. Thus, there is a single path between any pair of vertices in $\theta_k$ up to a power of cycles. 

    Applying Lemma~\ref{lem:theta-is-Post} and Proposition~\ref{Prop:Post-is-Thin} to $\Theta_k$ for $k>2$ gives the result.
\end{proof}
The following technical result gives a condition for paths between two vertices in the bridge quiver to be minimal.

\begin{lem} \label{lem:minimal-paths-in-theta} 
Let $g=\alpha_1\cdot \alpha_n$ be a path in $\Theta_k$ where $\alpha_i:(a_i,b_i) \to (a_{i+1},b_{i+1}) $ is an arrow between vertices $(a_i,b_i) , (a_{i+1},b_{i+1}) \in [k+1] \times [k+1]$ for each $1 \leq i \leq n$. If $g$ satisfies one of the following conditions, then $g$ is a minimal path from $(a_1,b_1) \to (a_{n+1},b_{n+1})$ in $A_{\Theta_k}$.
\begin{enumerate}
    \item If $a_1< a_{n+1}$ then $a_i < a_{i+1}$ for all $1 \leq i \leq n$.
    \item If $a_1 > a_{n+1}$ then $a_i > a_{i+1}$ for all $1 \leq i \leq n$.
    \item If $b_1< b_{n+1}$ then $b_i < b_{i+1}$ for all $1 \leq i \leq n$.
     \item If $b_1 > b_{n+1}$ then $b_i > b_{i+1}$ for all $1 \leq i \leq n$.
\end{enumerate}

\end{lem}
\begin{proof}
    Let $g$ be such a path. Then $g$ doesn't contain any arrows in a given direction determined by the hypothesis. By construction, any path equivalent to $g$ in the dimer algebra $A_{\Theta_k}$ will also be missing an arrow in this direction and so can't contain a cycle. The result then follows from Lemma~\ref{lem:theta-is-thin}. 
\end{proof} 

\begin{Rem}\label{rem:minimal-3dir}
    The conditions (1)-(4) of Lemma~\ref{lem:minimal-paths-in-theta} are satisfied by any path only using three out of the four directions N,E,S,W on the underlying lattice (where a diagonal path travels in two directions). Therefore by Lemma~\ref{lem:minimal-paths-in-theta} any such path is minimal in the associated dimer algebra.
\end{Rem}

%
\begin{notation}\label{notation:bridge-paths}
The following notations for paths in the bridge quiver will be useful in Section~\ref{sec:annulus}, when we need to refer to arrows and paths along the boundary of the bridge quiver $\Theta_k$ (see Figure ~\ref{fig:theta-arrows}). 

\begin{itemize}
\item $r_i:(i,k) \to (i+1,k)$ for $1 \leq i\leq k-1$
\item $r_k: (k,k) \to (k+1,k+1)$
\item $z_i: (k+1,k-i) \to (k+1, k-(i+1)) $ for $-2 \leq i \leq k-3$
\item $s_i: (k-i,2) \to (k-(i+1)),2)$ for $-1 \leq i \leq k-3 $
\item $s_k : (2,2) \to (1,1)$
\item $w_i: (1,i) \to (1,i+1)$ for $1 \leq i \leq {k-1}$.
\end{itemize}

\begin{figure}
    \centering
\includegraphics[width=1.15\linewidth]{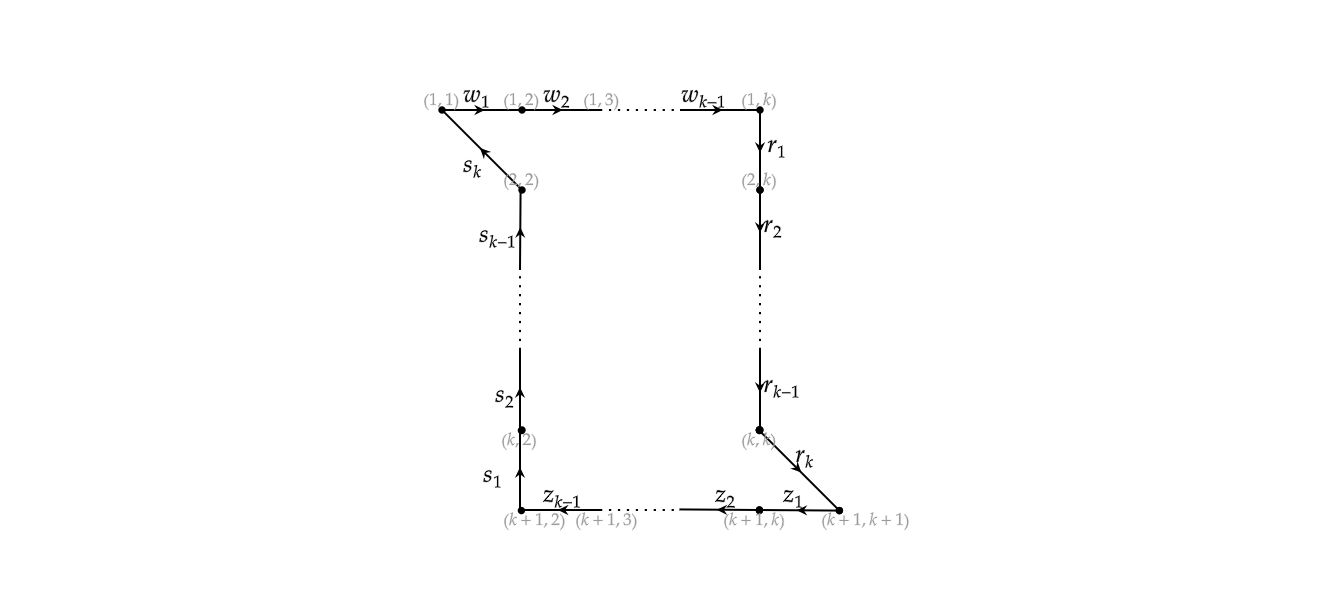}
    \caption{Labelled boundary arrows of $\Theta_k$}
    \label{fig:theta-arrows}
\end{figure}

We will also frequently use  
\begin{itemize}
\item $r:=r_1r_2\cdots r_k$
\item $s:=s_1s_2\cdots s_k$
\item $\widehat{r}:=\widehat{r_k}\widehat{r_{k-1}}\cdots \widehat{r_1}$
\item $\widehat{s}:=\widehat{s_k}\widehat{s_{k-1}}\cdots \widehat{s_1}$ 
\end{itemize} 
for the composition of all $k$ arrows $r_i$,  $s_i$, $\widehat{r_i}$ or $\widehat{s_i}$ respectively. 
\end{notation}

The boundary of $\Theta_k$ is a clockwise oriented cycle. Let $\gamma$ be an arrow of the boundary of $\Theta_k$, starting at a vertex $a$ and ending at a vertex $b$ of $\Theta_k$. Recall that $\hat{\gamma}$ denotes the path from $b$ to $a$ completing $\gamma$ to a cycle 
$\gamma\hat{\gamma}$ which forms the boundary of the unique face incident with $\gamma$ (cf. Notation~\ref{not:hat-and-u}). 

With this, 
the following notation makes sense: 

If $g=\gamma_1,..,\gamma_t:a \to b$ 
is a path composed of boundary arrows $\gamma_i$ in $\Theta_k$, we write 
$\widehat{g}:= \widehat{\gamma_t} \widehat{\gamma_{t-1}}...\widehat{\gamma_1}: b \to a$ for the path in the reverse direction formed by the complements.

For $1\le a<k$ and $1\le l\le k-a-1$ we use the following abbreviations for paths in the $z_i$s and $w_i$s, along the northern/southern boundary of the bridge quiver (as indicated by the coordinates defining $\Theta_k$):  
\begin{itemize}
    \item 
    $w^l(a):=w_a w_{a+1}\cdots w_{a+l-1}: (1,a)\to (1,a+l)$ 
    \item 
    $z^l(a):=z_az_{a+1}\cdots z_{a_l-1}: 
    (k+1,(k+2)-a) \to (k+1,k+2-(a+l))$
\end{itemize}

\section{Gluing dimer quivers for annuli}\label{sec:annulus}


In this section, we will glue the dimer quiver of a $(k,n)$-diagram and the bridge quiver $\Theta_k$ along two sets of arrows to obtain a dimer quiver on the annulus which in turn corresponds to a a weak Postnikov diagram on the new surface. We show that this diagram is in fact of degree $k$ - permuting the vertices along the boundary in a homogeneous way.  

Recall that if $Q$ is a dimer quiver on an arbitrary surface, 
we can associate a collection of strands to it which we denote by $D_Q$ (Remark~\ref{rem:Q-gives-D}). 

\begin{Def}\label{def:type-any-surface}
Let $D$ be a weak Postnikov diagram on a surface $\Sigma$. 
Label all endpoints of the strands of $D$ clockwise around each boundary component: if $B$ is a connected component of the boundary of $\Sigma$, label the endpoints by $1_B,\dots, (n_B)_B$ where $n_B$ is the number of endpoints of strands on $B$. Let $N:=\sum_B n_B$. 
We say that $D$ has {\em degree $k$} if all strands of $D$ start and end at the same boundary and if the permutation of 
$N$ is of the form 
$i_B\mapsto (i+k)_B$ (reducing modulo $n_B$) 
for every component. 
\end{Def}

\begin{Th}\label{Thm: glued-annulus-is-dimer} Let $D$ be a $(k,n)$-diagram in the disk with $n\geq 2k$ and $Q_D$ its dimer quiver. Let $Q=Q_D\cup \Theta_k$. 

Let $I_1$ and $I_2$ be two disjoint sets of $k$ consecutive arrows along $Q_D$ and let $J_1$ and $J_2$ be the sets $J_1=\{s_1,..,s_k\}$ and $J_2=\{r_1,...,r_k\}$ of consecutive arrows along the boundary of $\Theta_k$.
Then
\begin{enumerate}
    \item $Q_{I_1\equiv J_1,I_2\equiv J_2}$ is a dimer quiver on an annulus and
    \item The collection 
$D_{Q_{I_1\equiv J_1,I_2\equiv J_2}}$ is a weak Postnikov diagram of degree $k$.
\end{enumerate} 

%
%
\end{Th}

\begin{proof}
    (1) We know by Lemma~\ref{lm:gluing-makes-dimer} that $Q_{I_1\equiv J_1,I_2\equiv J_2}$ is a dimer quiver. The choice of $I_1$ and $I_2$ ensures that the underlying surface of $Q_{I_1\equiv J_1,I_2\equiv J_2}$ is an annulus.

    (2) By Lemma~\ref{lem:dimer-gives-weak-diagram}, the collection $D_{Q_{I_1\equiv J_1,I_2\equiv J_2}}$ of strands is a weak Postnikov diagram. It remains to compute the degree of this diagram. We will show that the strand diagram results in a degree $k$ permutation of boundary vertices on a fixed boundary component, with the case for the other boundary component being symmetric. 
    
    Strands in the glued dimer quiver $Q_{I_1\equiv J_1,I_2\equiv J_2}$ arise from composing each strand whose target is in an arrow in $I_t$ with the strand whose source is the glued arrow in $J_t$ (and similarly for strands whose source is in $I_t$ and target is in $J_t$) for $t\in \{1,2\}$. Consequently the degree of the strand permutation on $D_{Q_{I_1\equiv J_1,I_2\equiv J_2}}$ can be calculated by composing the permutation on $D$ with the strand permutation on $\Theta_k$ along the corresponding seams $I_1\equiv J_1$ or $I_2\equiv J_2$.

    For each $1 < m < k$, the strand starting at $j_{1_{k-m}}$ in the straightened strand diagram from Remark~\ref{rem:bridge-strands} has an $L$-shape and can be seen to have target $j_{1_{k-m}} + 2m-1 $ on the boundary of $\Theta_k$. The strand in $D_Q$ with target $i_{1_m}$ has source $i_{1_m} -k$. The resulting composition has the effect $+ (k-m) + m$ on the outer boundary component. If $m=k$, then the strand in $\Theta_k$ with source $j_{1_1}$ is straight with target $j_{2_k}$. The effect of the composition of the strand in $D_Q$ with target $j_{1_1}$, the strand in $\Theta_k$ from $j_{1_1}$ to $j_{2_k}$ and the strand in $D_Q$ with source $i_{2_k}$ on the outer boundary component is therefore $+1 + (k-2) + 1$. The other seam and boundary is completely similar.
    
\end{proof}
\begin{figure}
\centering
\includegraphics[scale=.8]{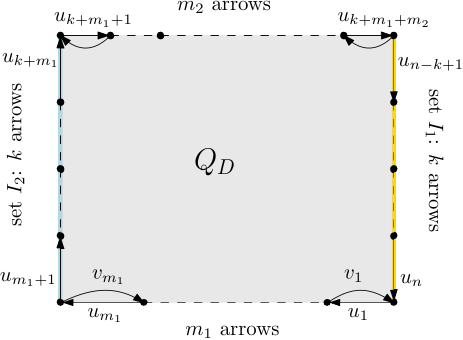}
\hskip.4cm
\includegraphics[scale=.8]{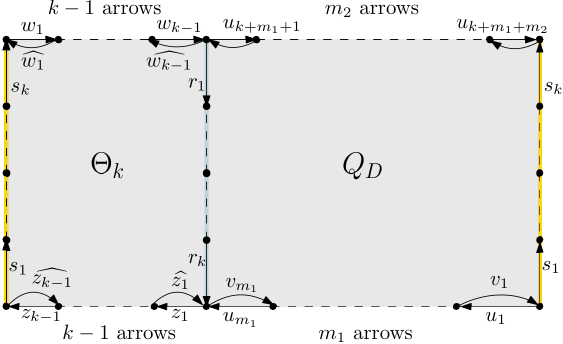}
\caption{The quivers $Q_D$ and $Q_{I_1\equiv J_1,I_2\equiv J_2}$. 
}
\label{fig:Q-I-J}
\end{figure}

\begin{Rem}\label{rem:size-boundary} 
 The quiver $Q_{I_1\equiv J_1,I_2\equiv J_2}$ has $n-2k$ vertices (or arrows) on its boundary. The partition of these 
boundary vertices depends on the choices of the sets $I_1$, $I_2$, more precisely, on the promixity of the two sets. Let $m_1 \geq 0$ (and $m_1<n$) denote the number of consecutive arrows on the boundary of $Q_D$ between the last arrow of $I_1$ and the first arrow of $I_2$, let $m_2\geq 0$ denote the number of consecutive arrows between the last arrow of $I_2$ and the first arrow of $I_1$, so that $n=m_1+m_1+2k$. These sets of $m_1$ (respectively $m_2$) arrows lie on the two different components of the boundary of the resulting quiver. We refer to them as the outer and inner boundary components respectively. 
See Figure~\ref{fig:Q-I-J} for an illustration.
\end{Rem}


We illustrate the theorem with two examples, the first one in the case $k=2$ where our construction recovers the decomposition of a triangulation of the annulus into a triangulation of the disk used in the proof of~\cite[Proposition 13.5]{BKM} as a special case. 

\begin{Ex} In the following example, $D$ is the dimer model arising from a triangulation of the disk (and therefore corresponding to a $(2,n)$-diagram). The arcs of the triangulation are drawn as dashed, in purple. Two sets $I_1$ and $I_2$ of two arrows along the boundary of $\rho(Q_D)$ are chosen, indicated in red and blue. 
The sets $J_1$ and $J_2$ of arrows of the bridge quiver each are the two arrows on the left/on the right, respectively. 
Here, $m_1=2$ and $m_1=0$. \vskip.2cm
\includegraphics[width=0.98\linewidth]{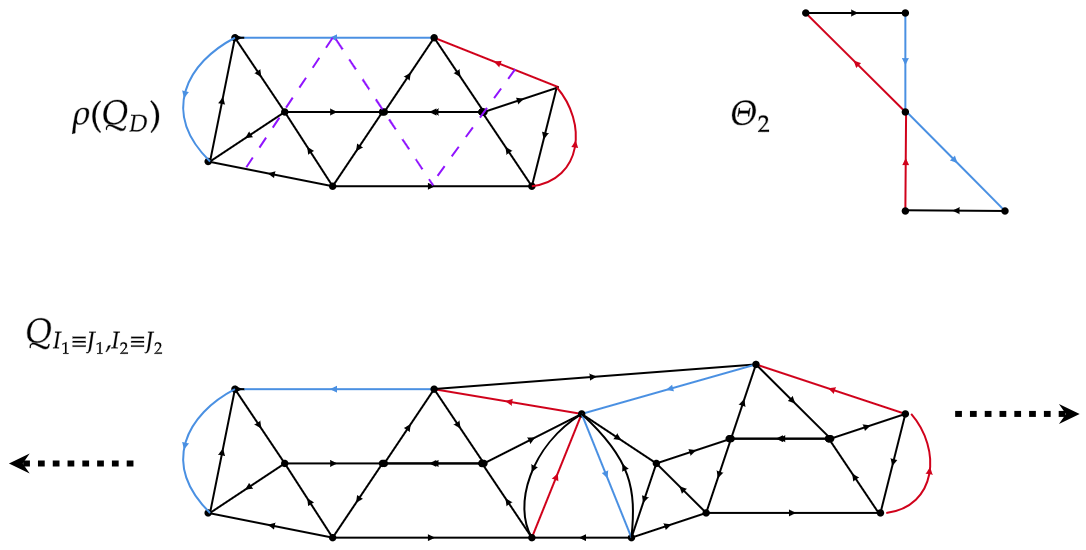} 
\end{Ex}


\begin{Ex}
In this example, $Q$ corresponds to a $(4,9)$-diagram $D$ in the disk. Two disjoint subsets of $4$ consecutive boundary arrows in $Q$ are chosen (illustrated in red and blue respectively) and are identified with the corresponding coloured arrows in $\Theta_k$. 
In the picture, $n=6$, $m_1=0$ and $m_2=1$. 
\vskip.2cm
\includegraphics[width=0.9\linewidth]{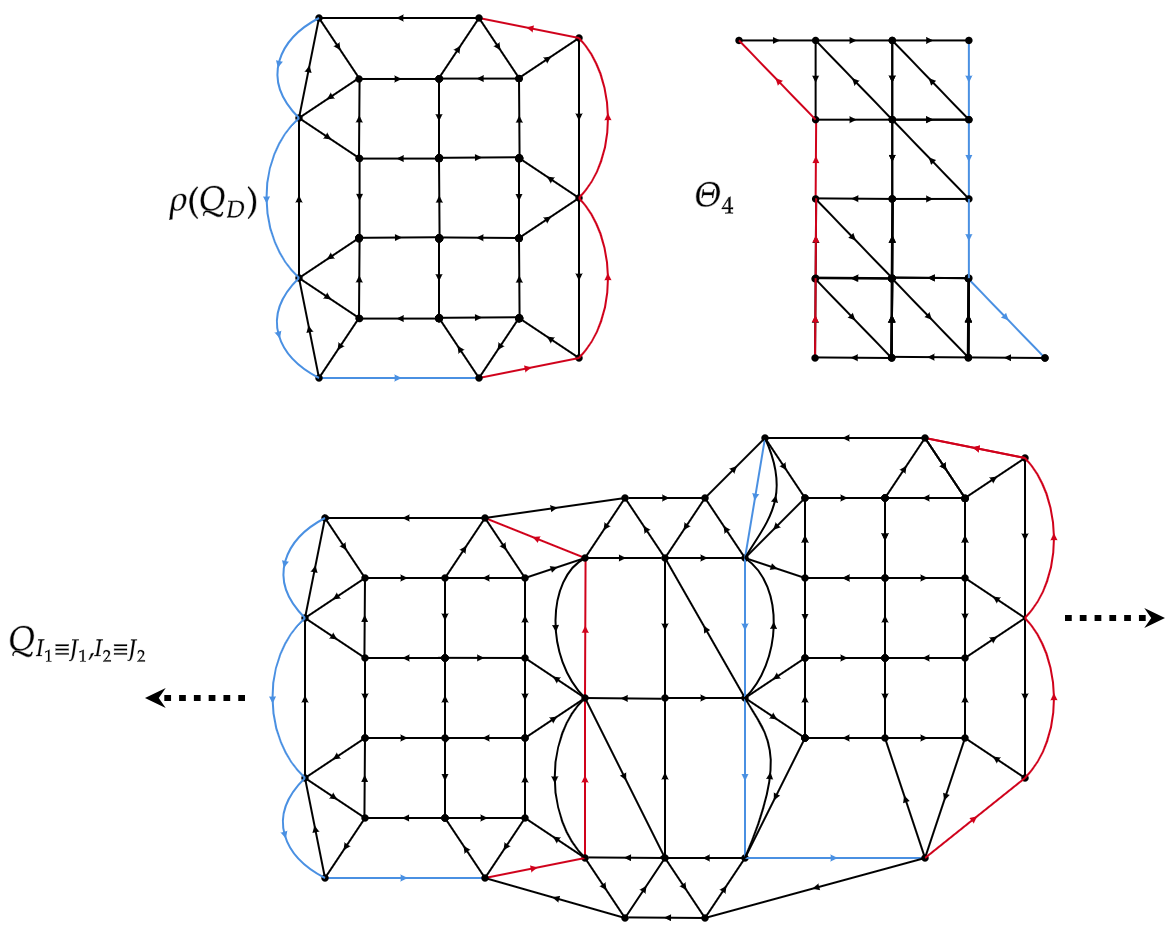} 

\end{Ex}


\begin{lem} \label{lem:thin-across-seams}
Let $D$ be any Postnikov diagram in the disk with $n\geq 2k$ marked points. Set $Q= Q_D \cup \Theta_k$ and let $I_1, I_2$ be disjoint sets of $k$-consecutive boundary arrows of $Q_D$, with $J_1=\{s_1,...,s_k\}, J_2=\{r_1,...,r_k\}$. Then both $D_{Q_{I_1 \equiv J_1}}$ and $D_{Q_{I_2 \equiv J_2}}$ are Postnikov diagrams on the disk.    
\end{lem}

\begin{proof}
By Lemma~\ref{lm:gluing-makes-dimer} $D_{Q_{I_1 \equiv J_1}}$ is a weak Postnikov diagram on the disk. Recall from Remark~\ref{rem:bridge-strands} that the set of strands in $D_{\Theta_k}$ with source in $J_1$ are mutually non-intersecting, and the set of strands in $D_{\Theta_k}$ with target in $J_1$ are also mutually non-intersecting. Consequently, the strands in $D_{Q_{I_1 \equiv J_1}}$ that cross the seam $I_1=J_1$ extend the strands in $D$ without introducing crossings, and so the global axioms of Definition~\ref{def:pos-diag} are satisfied. Similarly for $D_{Q_{I_2 \equiv J_2}}$. 
\end{proof}

Our goal is now to characterise the boundary algebra of the degree $k$ dimer quiver on the annulus from Theorem~\ref{Thm: glued-annulus-is-dimer}. 

Let $D$ be a $(k,n)$-diagram in the disk where $2k\le n$ and let $I_1$ and $I_2$ be disjoint sets of $k$ consecutive arrows along the boundary of the dimer quiver $Q_D$ as in Theorem~\ref{Thm: glued-annulus-is-dimer}. We fix a labelling of the boundary vertices of $\rho(Q_D)$ by $\{1,..,n\}$ so that the vertices incident with $I_1$ are labelled $n-k+1,n-k+2,\dots, 1$. Under this labelling and with the notation of Remark~\ref{rem:size-boundary} the vertices incident with $I_2$ are $m_1+1,m_1+2,\dots, m_1+k+1$.

Now let $u_i: i \to i+1$ and $v_i: i+1 \to i$ denote the minimal paths between consecutive boundary vertices in $\rho(Q_D)$ as in Lemma~\ref{lem:disk bdy}.

\begin{Rem}\label{rem:identifications}
In the glued quiver $Q_{I_1\equiv J_1,I_2\equiv J_2}$ we have by construction (recall that $n=m_1+m_2+2k$): 
\begin{itemize}
    \item 
    $s_1=v_n$, $s_2=v_{n-1},\dots, s_k=v_{n-k+1}$, 
    \item     $\widehat{s_1}=u_n$, $\widehat{s_2}=u_{n-1}, \widehat{s_k}=u_{n-k+1}$, 
    \item $r_1 = v_{m_1+k}, r_2 = v_{m_1+k-1},...,r_k = v_{m_1+1}$,
    \item $\widehat{r_1} = u_{m_1+k}, \widehat{r_2} = u_{m_1+k-1},...,
    \widehat{r_k}= u_{m_1+1}$. 
\end{itemize}

\end{Rem}

\begin{Rem}\label{rem:lifting-relations}
Note that any (connected) path in $\mathbb{C}(\rho(Q)\cup \Theta_k)$ may be regarded as a path in $\mathbb{C}Q_{I_1 \equiv J_1}$ and as a path in $\mathbb{C}Q_{I_2 \equiv J_2}$. Also any path in $\mathbb{C}Q_{I_1 \equiv J_1}$ or $\mathbb{C}Q_{I_2 \equiv J_2}$ that is solely contained in either $\mathbb{C}(\rho(Q))$ or $\mathbb{C} \Theta_k$ may be regarded as a path in $\mathbb{C}(\rho(Q)\cup \Theta_k)$. 
By Proposition~\ref{prop:algebra-glued}, $A_{Q_{I_1 \equiv J_1}}$ and $A_{Q_{I_2 \equiv J_2}}$ are quotients of $A_{\rho(Q)\cup \Theta_k}$, so if $g=h$ in 
$A_{\rho(Q)\cup \Theta_k}$ then equality holds for the corresponding paths in $A_{Q_{I_1 \equiv J_1}}$ and $A_{Q_{I_2 \equiv J_2}}$. 
Similarly, as $A_{Q_{I_1 \equiv J_1}, I_2 \equiv J_2}$ is a quotient of $A_{Q_{I_1 \equiv J_1}}$ and a quotient of $A_{Q_{I_2 \equiv J_2}}$, if $g=h$ in $A_{Q_{I_1 \equiv J_1}}$ or $A_{Q_{I_2 \equiv J_2}}$, then equality holds for the corresponding paths in $A_{Q_{I_1 \equiv J_1}, I_2 \equiv J_2}$.
\end{Rem}


In order to characterize the boundary algebra of the glued quiver we first determine a generating set for it. A natural candidate for the generating set are all the boundary arrows of $Q_{I_1\equiv J_1,I_2\equiv J_2}$ 
together with their completions and with the boundary idempotents. 
See Figure~\ref{fig:Q-I-J} for an illustration of this set of arrows and their completions. 

\begin{notation}\label{notation:setA}
Let $Q_{I_1\equiv J_1,I_2\equiv J_2}$ as above. We denote by 
$\mathcal A$ the following collection of elements of $\mathbb{C}Q_{I_1\equiv J_1,I_2\equiv J_2}$: 
\begin{align*}
\mathcal A:=&\{r,s\} \cup \{ w_i, \widehat{w_i},z_i, \widehat{z_i}: 1\leq i \leq k-1 \}   \\
&\cup \{  v_i, u_i:  1\leq i \leq m_1+1\} 
\\
&\cup \{ v_i, u_i:  m_1+k+1\leq i \leq m_1+k+m_2 \}  \\
    & \cup \{ e_i: i \text{ is a boundary vertex of } Q_{I_1\equiv J_1,I_2\equiv J_2} \}.
\end{align*}
\end{notation}

\begin{lem} \label{lem:bdy-generators} 
Let $Q_{I_1\equiv J_1,I_2\equiv J_2}$ and $\mathcal A$ be as above. Then we have: \\
The boundary algebra $eA_{Q_{I_1\equiv J_1,I_2\equiv J_2}}e$ is generated by $\mathcal{A}$. 
\end{lem}


\begin{proof}
We have to show that every path $g$ between boundary vertices of $Q_{I_1\equiv J_1,I_2\equiv J_2}$ can be expressed in terms of elements of $\mathcal A$.

First consider such paths $g$ which remain entirely in $\rho(Q_D)$ or entirely in the bridge $\Theta_k$. Such a path may be regarded as 
an element of either $A_{\rho(Q_D)}$ or of $A_{\Theta_k}$ respectively. In the first case there is a minimal path from $a$ to $b$ given by either $v^l$ or $u^l$ by Lemma~\ref{lem:disk bdy}. By Remark~\ref{rem:lifting-relations} there is a relation $v^lt^h= g$ (respectively $u^lt^h= g$) in $A_{Q_{I_1\equiv J_1,I_2\equiv J_2}}$. Under the identifications in Remark~\ref{rem:identifications}, $g= v^irv^jt^h$ (respectively $g= u^isu^jt^h$) where $i+j =l-k$.

In the second case $g$ is contained in $\Theta_k$ and we separately consider the placements of $a,b$ on the boundary components of 
$Q_{I_1\equiv J_1,I_2\equiv J_2}$. If both $a$ and $b$ lie on outer boundary of 
$Q_{I_1\equiv J_1,I_2\equiv J_2}$, then by Lemma~\ref{lem:minimal-paths-in-theta} there is a minimal path from $a$ to $b$ given by either $w^l$ or $\widehat{w^l}$. If $a$ lies on the outer boundary and $b$ lies on the inner boundary, then $a=(1,i), b = (k+1,k+2-j)$ for some $i,j \in [k]$. By Lemma~\ref{lem:minimal-paths-in-theta} there is a minimal path from $a$ to $b$ given by $w^{k-i}rz^{j-1}$. The corresponding statements where the boundary components of $a,b$ are swapped follows by symmetry.

Let $g$ be a path between boundary vertices in $Q_{I_1 \equiv J_1, I_2 \equiv J_2}$. Note that we can express $g$ as a concatenation of seam crossings $g = \sigma_1\sigma_2,..,\sigma_q$ where the end vertex of $\sigma_i$ is the initial vertex of $\sigma_{i+1}$.
We now proceed by induction on the number of seam crossings in $g$.

Suppose $g$ crosses the seam $I_\delta \equiv J_\delta$ twice consecutively for some $\delta \in \{1,2\}$. That is, for some $1 \leq i < q$ $\sigma_i = g_0g_1g_2$ and $\sigma_{i+1}=h_0h_1h_2$ cross $I_\delta \equiv J_\delta$ and $g_2h_0:x\to y$ is a path with no seam crossings. Then $g_2h_0:x \to y$ may be regarded as a path completely contained in $A_{\rho(Q)\cup \Theta_k}$. 

By Lemma~\ref{lem:disk bdy} a minimal path $\gamma: x \to y$ in $A_{Q_D}$ is given by either $u_xu_{x+1}...u_{y}$ or $v_xv_{x+1}...v_{y}$ (the length of such a path is bounded above by both $k$ and $n-k$). In either case, $\gamma$ is contained in $I_\delta \equiv J_\delta$ and is therefore equivalent (by Remark~\ref{rem:identifications} in $A_{Q_{I_\delta \equiv J_\delta}}$ to a path from $\gamma^\prime:x\to y$ that is completely contained in $\Theta_k$. By Lemma~\ref{lem:minimal-paths-in-theta} the path $\gamma^\prime$ is minimal in $A_{\Theta_k}$. As a result, the path $g_2h_0 :x\to y$ viewed as a path in $A_{Q_{I_\delta \equiv J_\delta}}$ satisfies the relation $g_2h_0^\prime = \gamma t^d = \gamma^\prime t^d$ for some $d\geq 0$ where $\gamma t^d = \gamma^\prime t^d$ does not cross the seam $I_\delta \equiv J_\delta$. Iterating this procedure, we may assume $g$ alternates between crossing the seam $I_1 \equiv J_1$ and the seam $I_2 \equiv J_2$. We proceed by induction on the number of seam crossings.

Let $\sigma_1 = g_0g_1g_2$ be the first seam crossing in $g$. Suppose first that $g_0:a \to v_0$ is contained in $\Theta_k$. If $q=1$ then $g= \sigma_1h_0$ for some path $h_0$ from the endpoint of $g_2$ to $b$ that is completely contained in $\rho(Q_D)$. Otherwise if $q>1$ then  $\sigma_1\sigma_2 = g_0g_1g_2h_0h_1h_2$ where $g_2h_0$ is a path with no crossings between boundary vertices of $\rho(Q_D)$. In either case the path $g_1g_2$ may be viewed as a path in $\rho(Q_D)$. Furthermore, the endpoints of $g_1g_2h_0$ are boundary vertices $x$ and $y$ of $\rho(Q_D)$. As we assumed $g$ cannot cross the same seam twice consecutively, $y$ lies on a boundary vertex of $A_{\rho(Q_D)}$ that is not on same seam as $x$. By Lemma ~\ref{lem:disk bdy} $g_1g_2h_0$ is equivalent to either $v^lt^d$ or $u^lt^d$ in $A_{\rho(Q_D)}$. In the first case we can express $v^lt^d$ in the form $v^iv^jt^d$ where $v^i:x \to w$ is the intersection of $v^lt^d$ with $I_2 \equiv J_2$ and $i+j=l$.   

We can now write $g$ as $g= g_0v^iv^jt^dh_1h_2\sigma_3,...,\sigma_q$ (where $h_1h_2\sigma_3,...,\sigma_q$ is trivial if $q=1$). The subpath $g_0v^i$ is a path with no crossings between boundary vertices of $A_{Q_{I_1\equiv J_1,I_2\equiv J_2}}$ and therefore the base case applies. The subpath $v^jt^dh_1h_2\sigma_3,...,\sigma_q$ has $q-1$ crossings and therefore satisfies the claim by induction. 

In the second case we can similarly express $u^lt^d$ in the form $u^iu^jt^d$ where $u^i:x \to w$ is the subpath that intersects the seam $I_2 \equiv J_2$ and $i+j=l$. Then $g= g_0u^iu^jt^dh_1h_2\sigma_3,...,\sigma_q$ (again $h_1h_2\sigma_3,...,\sigma_q$ is trivial if $q=1$). The subpath $g_0u^i$ is between boundary vertices of $A_{Q_{I_1\equiv J_1,I_2\equiv J_2}}$ with no crossings, $u^jt^d$ clearly satisfies the claim, as does $h_1h_2\sigma_3,...,\sigma_q$ by induction as it crosses $q-1$ seams. 

Now suppose $g_0:a \to v_0$ is contained in $\rho(Q_D)$. The subpath $g_0g_1$ from $a \to x$ is also completely contained in $\rho(Q_D)$ and between boundary vertices of $a$ and $x$ of $\rho(Q_D)$. By Lemma ~\ref{lem:disk bdy} we can express $g_0g_1$ as $v^lt^d$ or $u^lt^d$ in $A_{\rho(Q_D)}$. In the first case we can express $v^lt^d$ in the form $v^iv^kv^{m_2}v^jt^d$ where $v^i$ is the subpath $a\to k+m_1$, $v^k$ is the subpath $k+m_1\to m_1$, $v^{m_1}$ is the subpath  $m_1 \to 1$, $v^j$ is the subpath $1\to**$, and $i+j=k-l-m_2$. From Remark~\ref{rem:identifications} we see that $v^k:k+m_1\to m_1$ is identified with $r$ in $A_{Q_{I_1\equiv J_1,I_2\equiv J_2}}$.

Now we have $g= v^irv^{m_2}v^jt^dg_2\sigma_2,..,\sigma_q$. The subpath $v^irv^{m_2}$ satisfies the conditions of the claim, while the subpath $v^jt^dg_2\sigma_2,..,\sigma_q$ has $q-1$ seam crossings and so the result holds by induction.

In the second case, we can express $u^lt^d$ in the form $u^iu^jt^d$ where $u^i$ is the subpath $a \to n-k$ and $u^j$ is the subpath $n-k \to y$ that intersects the seam $I_1\equiv J_1$ and $i+j=l$. 

We can therefore express $g $ as  $g= u^iu^jt^dg_2\sigma_2,...,\sigma_q$. The subpath $u^i$ satisfies the claim and the subpath $u^jt^hg_2\sigma_2,...,\sigma_q$ has $q-1$ seam crossings and so the result holds by induction.

\end{proof}


\begin{notation}
We define a quiver on an annulus: 
For $k \geq 2$ and $m_1,m_2 \geq 0$, let $\Gamma=\Gamma_{m_1,m_2,k}$ be the quiver with $2k-2+m_1+m_2$ vertices arranged as follows. 
$\Gamma$ has $m_2 +k-1$ vertices 
on the outer boundary, arranged clockwise. It has $m_1+k-1$ vertices 
on the inner boundary, also also arranged clockwise. 
The arrows of $\Gamma_{m_1,m_2,k}$ are given by $y_i: i \to i+1, x_i: i+1 \to i$ on the outer boundary, for $1\le i\le k+m_2-1$, $\overline{x_i}:\overline{i}\to\overline{i+1}$, $\overline{y_i}:\overline{i+1} \to \overline{i}$ on the inner boundary, $1\le i\le k+m_1-1$ and arrows $\tilde{s}:\overline{1} \to 1$ and $\tilde{r}:k \to \overline{k}$ between the boundary components. See Figure~\ref{fig:quiverGamma}. 
\begin{figure}
\centering
\includegraphics[scale=.8]{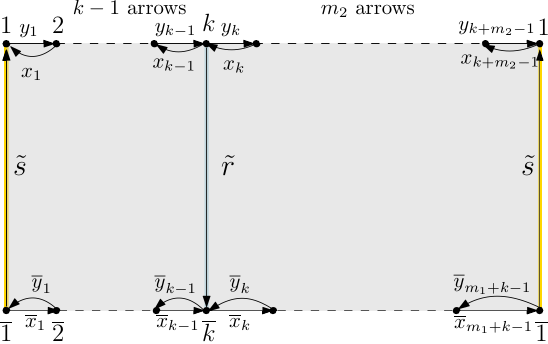}
\caption{The quiver $\Gamma_{m_1,m_2,k}$}
\label{fig:quiverGamma}
\end{figure}
\end{notation}

\begin{Ex} The following is an illustration of the quiver $\Gamma_{8,4,3}$
\begin{center}
\includegraphics[width=0.5\linewidth]{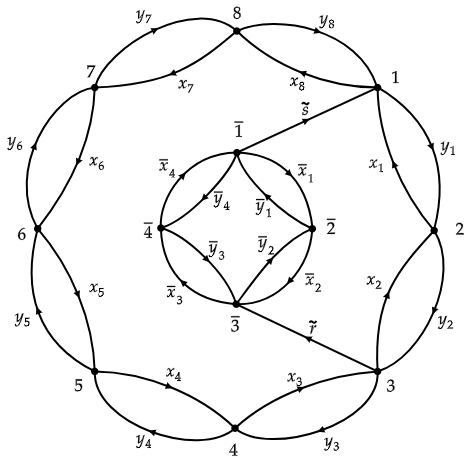} 
\end{center}
\end{Ex}

We will drop the indices and refer to the arrows as $x,y,\overline{x},\overline{y}$ whenever the indices are clear from composition.

\begin{Def}\label{def:ideal-annulus}
Let $\Lambda_{m_1,m_2,k}$ be the quotient of the path algebra $\mathbb{C}\Gamma_{m_1,m_2,k}$ by the ideal generated by the following relations. 
\[
\begin{array}{cll}
(1) & y_ix_i - x_{i-1}y_{i-1} & i=1,\dots, k+m_2-1 \\ 
(2) & \overline{x}_i\overline{y}_i - \overline{y}_{i-1}\overline{x}_{i-1} & i=1,\dots, k+m_1-1  \\
 & & \\
(3) & x^k - y^{j_1}\tilde{r}\overline{y}^{2k+m_1-2}\tilde{s}y^{k+m_2-j_1-2} & \mbox{($k+m_2-1$ relations)}\\
(4) & \overline{x}^k - \overline{y}^{j_2}\tilde{s}y^{2k+m_2-2}\tilde{r}\overline{y}^{k+m_1-j_2-2} & \mbox{($k+m_1-1$ relations)} \\ 
 & \\
(5) & \tilde{s}-\overline{y}^{m_1+k-1}\tilde{s}{y}^{m_2+k-1} \\
(6) & \tilde{r}-y^{m_2+k-1} \tilde{r} \overline{y}^{m_1+k-1} \\ 
 & \\ 
(7) & x_{k-1}y_{k-1}\tilde{r}-\tilde{r}\overline{y}_{k-1}\,\overline{x}_{k-1}\\
(8) & \overline{x}_{m_1+k-1}\,\overline{y}_{m_1+k-1}\tilde{s}-\tilde{s}y_{k+m_2-1}x_{k+m_2-1}\\    
\end{array}
\]


\noindent
where there are relations of type (1) and (3) for each of the $k+m_2-1$ vertices on the outer boundary and where there are relations of type (2) and (4) for each of the $k+m_1-1$ vertices on the inner boundary. 
The exponent $j_1\ge 0$ in (3) is minimal such that the path $y^{j_1}$ ends at the starting point of $\tilde{r}$ and the exponent $j_2\ge 0$ in (4) is the minimum such that the path $\overline{y}^{j_2}$ ends at the starting point of $\tilde{s}$. 
Indices are reduced modulo $k+m_2-1$ on the outer boundary and modulo $k+m_1-1$ on the inner boundary. 
\end{Def}

Note that replacing $\tilde{r}$ and $\tilde{s}$ in the presentation of $\Lambda_{m_1,m_2,k}$ by any pair of arrows linking the two double cycles of $m_2+k-1$ and $m_1+k-1$ arrows in a way to keep the distances between these two arrows fixed (the distance between the starting point of $\tilde{s}$ and the endpoint of $\tilde{r}$ as well as the distance between the starting point of $\tilde{r}$ and the endpoint of $\tilde{s}$) and modifying relations (3)-(8) accordingly gives an isomorphic algebra. Further, similar presentations may be derived from an arbitrary choice of $\tilde{r},\tilde{s}$ between the two double cycles.  

\begin{Th} \label{thm:main}
Let 
$Q_{I_1\equiv J_1,I_2\equiv J_2}$ a dimer quiver on the annulus as in  Theorem~\ref{Thm: glued-annulus-is-dimer} and $e$ the idempotent of  its boundary vertices. 
Then we have an isomorphism 
\[
eA_{Q_{I_1\equiv J_1,I_2\equiv J_2}}e \cong 
\Lambda_{m_2,m_1,k}
\]
\end{Th}

\begin{proof}
Recall that 
$Q$ denotes the quiver $\rho(Q_D)\amalg\Theta_k$ with $I_1,I_2$ two disjoint sets of $k$ connected arrows on the boundary of $Q_D$. 

Let $\mathcal{R}_1$ and $\mathcal{R}_2$ be the relations derived from the potentials on $\rho(Q_D)$ and $\Theta_k$ respectively and let $R_i=\langle\mathcal{R}_i\rangle$ be the corresponding ideals, i.e. $A_{\rho(Q_D)} = \mathbb{C}(\rho(Q_D))/R_1$ and $A_{\Theta_k} = \mathbb{C}\Theta_k/R_2$. 
Throughout the proof we write 
$\Gamma$ for $\Gamma_{m_2,m_1,k}$ and 
$\Lambda$ for $\Lambda_{m_2,m_1,k}$ to abbreviate and we often use 
$Q_{I\equiv J}$ as a short notation for $Q_{I_1\equiv J_1,I_2\equiv J_2}$.
Let $\mathcal A$ be the collection of paths of $Q_{I\equiv J}$ from Notation~\ref{notation:setA}. 

Our goal is to define a map $\mathbb{C}\Gamma\to e(A_{Q_{I\equiv J}})e$ and to show that its kernel is the ideal given by the relations (1)-(8) of Definition~\ref{def:ideal-annulus}. 

We first define $\psi:\mathbb{C}\Gamma\to \mathbb{C}Q_{I\equiv J}$ on arrows of $\Gamma$ as follows (recalling that we write $s$ for the composition $s_1\cdots s_k$ and $r$ for $r_1\cdots r_k$, Notation~\ref{notation:bridge-paths}).

\[
\begin{array}{lcl}
\psi(\tilde{r}) & = & r \\
\psi(\tilde{s}) & = & s \\
 & & \\
\psi(x_i) & = & 
    \left\{ 
    \begin{array}{ll}
     \widehat{w}_i  & \ \mbox{ if $1\le i\le k-1$} \\
     v_{m_1+i-1}  & 
     \ \mbox{ if $k\le i\le k+m_2-1$} 
    \end{array}\right.\\
 & & \\
\psi(y_i) & = & 
    \left\{ 
    \begin{array}{ll}
     w_i   & \ \mbox{ if $1\le i\le k-1$} \\
     u_{m_1+i-1}  & 
     \ \mbox{ if $k\le i\le k+m_2-1$} 
    \end{array}\right.\\
 & & \\
\psi(\overline{x_i}) & = & 
    \left\{ 
    \begin{array}{ll}
     \widehat{z}_{k-i} & \ \mbox{if $1\le i\le k-1$} \\
     v_{m_1+k-i}  & 
     \ \mbox{if $k\le i\le k+m_1-1$}    \end{array}\right.  \\
 & & \\
\psi(\overline{y_i}) & = & 
    \left\{ 
    \begin{array}{ll}
     z_{k-i}  & \mbox{if $1\le i\le k-1$} \\
     u_{m_1+k-i}  & 
     \mbox{if $k\le i\le k+m_1-1$}  \end{array}\right.\\
\end{array}
\]

On vertices $\psi$ is determined from the mappings on arrows. 
We extend $\psi$ to all paths in $\Gamma$ and $\mathbb{C}$-linearly to the algebra 
$\mathbb{C}\Gamma$. 

We then compose $\psi$ with the quotient map $\mathbb{C}Q_{I\equiv J}\to A_{Q_{I\equiv J}}$
to the dimer algebra and call the resulting map 
$\varphi$. 

By construction, the image of $\varphi$ lies in $eA_{Q_{I\equiv J}}e$: the map $\psi$ sends arrows of $\Gamma$ to paths between boundary vertices of $Q_{I\equiv J}$ in such a way that the latter agrees with the vertices of $\Gamma$. 
The map $\varphi$ is a homomorphism of algebras.
Since $\mathcal{A}$ is a 
generating set for $e(A_{Q_{I\equiv J}})e$ by Lemma~\ref{lem:bdy-generators}, $\varphi$ is surjective. 
By the first isomorphism theorem, we thus have 
$e(A_{Q_{I\equiv J}})e\cong \mathbb{C}\Gamma/\ker\varphi$.
To see that the latter is equal to $\Lambda$, it remains to show that the kernel of $\varphi$ is equal to the ideal of $\mathbb{C}\Gamma$ given by the relations (1)-(8) of Definition~\ref{def:ideal-annulus}. 

We first express $\ker\varphi$ in terms of the ideals $R_1$, $R_2$ and of the relations arising from the seam. Let $I_1=\{i_{1,1},\dots, i_{1,k}\}$ and $I_2=\{i_{2,1},\dots, i_{2,k}\}$ the arrows of $Q_D$ giving the seam together with the sets $J_1=\{s_1,\dots, s_k\}$ and $J_2=\{r_1,\dots, r_k\}$ of arrows of $\Theta_k$. We define 
\[
\mathcal{W}_1:=
\left\{
\begin{array}{lcl} 
s_{m}- \rho(i_{1,m}) & : &  
1 \leq m \leq k,  \\
\widehat{s_{m}}-\widehat{\rho(i_{1,m})} & : & 1 \leq m \leq k,  \\
e_{h(s_m)} -e_{h(i_{1,m})} & : & 1 \leq m \leq k,  \\
     e_{t(s_k)} -e_{t(i_{1,k})} 
\end{array}
\right\}
\]
and 
\[
\mathcal{W}_2:=
\left\{
\begin{array}{lcl}
r_{m} - \rho(i_{2,m}) 
& : & 1 \leq m \leq k,  \\ 
\widehat{r_{m}}-\widehat{\rho(i_{2,m})} & : & 1 \leq m \leq k,  \\
      e_{h(r_m)} -e_{h(i_{2,m})} & : & 1 \leq m \leq k,  \\
     e_{t(r_k)} -e_{t(i_{2,k})} 
\end{array} \right\}. 
\]
Note that the first term in these sets arises from identifying the arrows of $I_1$ with the arrows of $J_1$ and the arrows of $I_2$ with the arrows of $J_2$, respectively. The second term arises from the potential in $Q_{I\equiv J}$ along the arrows of the seam. 
%
and let $W_1=\langle \mathcal{W}_1\rangle$, 
$W_2=\langle\mathcal{W}_2\rangle$
and 
$R_3=\langle\mathcal{W}_1,\mathcal{W}_2\rangle$. These ideals of $\mathcal{C}Q_{I\equiv J}$ are given by the terms in the potential arising from the seam $I_1\equiv J_1$, from the seam $I_2\equiv J_2$ and from both seams, respectively. 


By abuse of notation, 
we consider $W_1,W_2,R_1,R_2$ and $R_3$ as ideals of $\mathbb{C}Q$.

\vskip.5cm

Since $Q_{I\equiv J}$ can be obtained from $Q_{I_1\equiv J_1}$ by gluing $I_2\equiv J_2$ and from 
$Q_{I_2\equiv J_2}$ by gluing 
$I_1\equiv J_1$, we can write  
$A_{Q_{I\equiv J}}$ as 
$A_{Q_{I_1\equiv J_1}}/W_2$ or as 
$A_{Q_{I_2\equiv J_2}}/W_1$ 
(Proposition~\ref{prop:algebra-glued}) and so 
$A_{Q_{I\equiv J}}=
A_Q/\langle\mathcal{ W}_1,\mathcal{W}_2\rangle$. 
Furthermore, as $Q$ is the 
disjoint union of $\rho(Q_D)$ and $\Theta_k$, we have 
$A_Q=\mathbb{C}Q/\langle \mathcal{R}_1,\mathcal{R}_2\rangle$ giving 
\[
\mathbb{C}Q/\langle\mathcal{R}_1,\mathcal{R}_2,\mathcal{R}_3\rangle
\]

\vskip.2cm

It remains to show that  $\ker(\varphi)$ is given by the relations (1)-(8). We check relations (1), (3), (5) and (7) hold as the rest is symmetric. 

\vskip.2cm

\noindent
\underline{Relations (1)}: 
\[
\varphi(y_ix_i-x_{i-1}y_{i-1})
 = 
\left\{
\begin{array}{ll}
w_1\widehat{w}_1 - v_{k+m_1+m_2}u_{k+m_1+m_2} & \mbox{if $i=1$,} \\
w_i\widehat{w_i} - \widehat{w}_{i-1}w_{i-1} & \mbox{if $1< i \le k-1$,} \\ 
u_{k+m_1+1}v_{k+m_1+1} - \widehat{w}_{k-1}w_{k-1} & \mbox{if $i=k$,} \\
u_{m_1+i-1}v_{m_1+i-1} - v_{m_1+i-2}u_{m_1+i-2} & \mbox{if $k< i\le k+m_1-1$.}
\end{array}
\right.
\]
The second and fourth of these 
cases are relations in $\mathcal{R}_2$ and $\mathcal {R}_1$, respectively. 
To see that the relation holds for $i=1$: 
\[
w_1\widehat{w}_1 \stackrel{\mathcal{R}_1}{=}
\widehat{\rho(i_{1,1})}i_{1,1}\stackrel{\mathcal{R}_3}{=}\widehat{s}_ks_k
\stackrel{\mathcal{R}_2}{=}v_{k+m_1+m_2}u_{k+m_1+m_2}
\]
For the relation in case $i=k$: 
\[
u_{k+m_1+1}v_{k+m_1+1}
\stackrel{\mathcal{R}_1}{=} 
\widehat{\rho(i_{2,k})}i_{2,k}\stackrel{\mathcal{R}_3}{=}r_1\widehat{r}_1
\stackrel{\mathcal{R}_2}{=}\widehat{w}_{k-1}w_{k-1}
\]
So in all cases, the term $x_iy_i-y_{i-1}x_{i-1}$ is in the kernel of $\varphi$. 

\vskip.2cm

\noindent
\underline{Relations (3):} 
The proof of these relations 
is in 
Lemma~\ref{lm:relation3} below.  

\vskip.2cm

\noindent
\underline{Relation (5):} 
We have to check  $\varphi(\tilde{s}-\overline{y}^{m_1+k-1}\tilde{s}{y}^{m_2+k-1})=0$. 

\[
\begin{array}{llcl}
\varphi(\tilde{s}) = & s & 
\stackrel{\mathcal{R}_3}{=} & v_nv_{n-1}...v_{n-k+1} 
\stackrel{\mathcal{R}_1}{=} u_1u_{2}...u_{k+m_1+m_2} \\ 
&& 
\stackrel{\mathcal{R}_3}{=} &  u_1u_{2}...u_{m_1}\hat{r}u_{k+m_1}u_{k+m_1+2}...u_{k+m_1+m_2} \\ 
&& \stackrel{\mathcal{R}_2}{=} &u_1u_{2}...u_{m_1}z_1z_2...z_{k-1}sw_1w_2...w_{k-1}u_{k+m_1}u_{k+m_1+2}...u_{k+m_1+m_2} \\
 & \\
&& = & \varphi(\overline{y}^{m_1+k-1}\tilde{s}{y}^{m_2+k-1})
\end{array}
\]


\vskip.2cm

\noindent
\underline{Relation (7):} 
We have to check  
$\varphi(x_{k-1}y_{k-1}\tilde{r}-\tilde{r}\,\overline{y}_{k-1}\,\overline{x}_{k-1})=0$. 
\[
\varphi(x_{k-1}y_{k-1}\tilde{r}) = r\widehat{w_{k-1}}w_{k-1} 
\]
As unit cycles are central in the dimer algebra we have

\begin{eqnarray*}
\widehat{w_{k-1}}w_{k-1}r  & {=} & 
r_1\widehat{r_1}r_1r_2\cdots r_k \\ 
 & = & 
 r_1r_2\widehat{r_2}r_2...r_k \\ 
 & = & \cdots \\ 
 & = & r_1r_2...r_k\widehat{r_k}r_k \\ 
  & = & r\widehat{r_k}r_k
\end{eqnarray*}

Under the gluing identifications this becomes
 \[
r\widehat{r_k}r_k \stackrel{\mathcal{R}_3}{=} r \rho(i_{2,k})\widehat{\rho(i_{2,k})}
 \]
where $\rho(i_{2,k})\widehat{\rho(i_{2,k})}$ is a unit cycle at the target of $r$. Therefore we obtain
 \[
r \rho(i_{2,k})\widehat{\rho(i_{2,k})} \stackrel{\mathcal{R}_1}{=} rv_{m_1}u_{m_1} = \varphi(\tilde{r}\,\overline{y}_{k-1}\,\overline{x}_{k-1}).
 \]

We have therefore shown that the generators of the ideal of relations defining $\Lambda$ belong to $ker(\varphi)$. As no other minimal, nontrivial compositions are possible in $\langle \mathcal{R}_1, \mathcal{R}_2, \mathcal{R}_3\rangle$ we have equality. 

Therefore we have the desired isomorphism. 
\end{proof}

\begin{figure}
    \centering    \includegraphics[width=7cm]{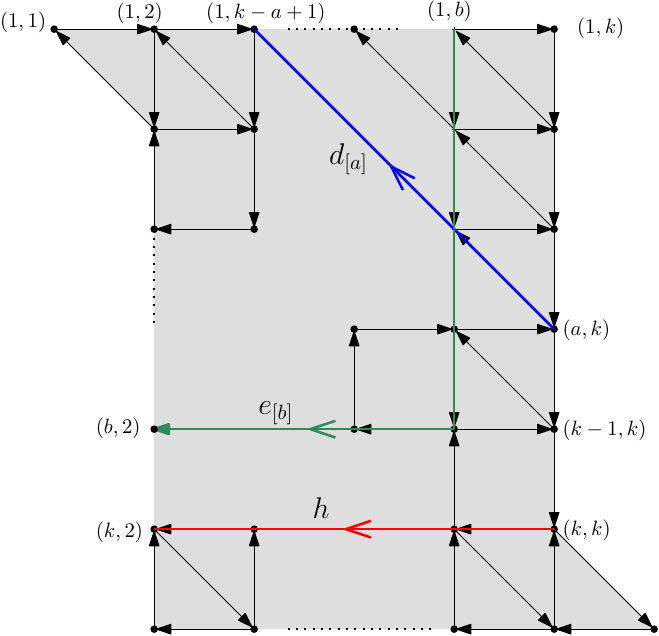}
    \caption{Shape of $\Theta_k$, paths $h,d_{[a]},e_{[b]}$.}
    \label{fig:bridge-shape}
\end{figure}

\begin{lem}\label{lm:relation3}
In the situation of Theorem~\ref{thm:main}, the $m_2+k-1$ relations 
\[
\varphi(x^k)=\varphi(y^{j_1}\tilde{r}\overline{y}^{2k+m_1-2}\tilde{s}y^{k+m_2-j_1-2})
\]
hold for any of the $m_2+k-1$ paths of type $x^k$ starting on the outer boundary.
\end{lem}

\begin{proof}
We write the term $x^k$ depending on its starting position between the starting point of $\tilde{r}$ and the ending point of $\tilde{s}$, as 
$x^{a_1}x^{a_2}x^{a_3}x^{a_4}$ with $a_i\ge 0$: 
the first $a_1$ arrows are from the set $\{x_1,\dots ,x_{k-1}\}$, the second $a_2$ arrows from 
the set $\{x_k,\dots, x_{k+m_2-1}\}$, the third $a_3$ arrows from the set $\{x_1,\dots ,x_{k-1}\}$ and the last $a_4$ arrows are again from the set $\{x_k,\dots, x_{k+m_2-1}\}$ (see Figure~\ref{fig:quiverGamma} 
for the labeling of the arrows in $\Gamma$). 
The exponents $a_i$ are all non-negative and add up to $k$. Furthermore, $a_1$ and $a_3$ are strictly smaller than $k$. We may assume that $a_4=0$ (if in the expression for $x^k$, we had $a_4>0$, we would have $a_2=a_1=0$. Such a case can be covered by $a_2>0$, with $a_4=0$). 

The triples $(a_1,a_2,a_3)$ are therefore one of the following: 
\[
\begin{array}{cll}
(i) & (a_1,m_2,k-a_1-m_2) & 
\mbox{where } 0<a_1<k, \  a_1+m_2\le k  \\
(ii) & (0,a_2,k-a_2) & \mbox{where } 0<a_2< k \\ 
(iii) & (0,k,0) & 
\end{array}
\]
Note that $m_2$ or $k-a_1-m_2$ may be $0$. We obtain 
\[ 
x^k = 
\left\{ 
\begin{array}{ll}
\overbrace{x_{a_1}x_{a_1-1}\cdots x_1}^{\small{\mbox{$a_1$ terms}}} 
\overbrace{x_{k+m_2-1} x_{k+m_2-2}\cdots 
x_{k}}^{\small{\mbox{$m_2$ terms}}}\overbrace{x_{k-1} x_{k-2}\cdots x_{a_1+m_2}}^{\small{\mbox{$k-a_1-m_2$ terms}}}
& \mbox{ in case (i)}, \\ 
 & \\
\overbrace{x_{k+a_2-1}x_{k+a_2-2}\cdots x_k}^{\small{\mbox{$a_2$ terms}}} 
\overbrace{x_{k-1} x_{k-2}\cdots x_{a_2}}^{\small{\mbox{$k-a_2$ terms}}} & \mbox{ in case (ii)}, \\ 
 & \\
\overbrace{x_{2k+b-1}x_{2k+b-2}\cdots x_{k+b}}^{\small{\mbox{$k$ terms}}} 
\hskip.5cm \mbox{with $0\le 
b\le m_2-k$} 
 & \mbox{ in case (iii)}. 
\end{array}
\right.
\]

The exponent $j_1$ in the term 
$(*):=y^{j_1}\tilde{r}\overline{y}^{2k+m_1-2}\tilde{s}y^{k+m_2-j_1-2}$ is equal to 
$k-1-a_1$ in the first case, to 
$m_2-a_2+k-1$ in the second case and to 
$m_2-k-b$ in the third case, giving 

\[ 
* = 
\left\{ 
\begin{array}{ll}
\overbrace{y_{a_1+1}y_{a_1+2}\cdots y_{k-1}}^{\mbox{\small{$k-1-a_1$ terms}}} \tilde{r}\,
\overline{y}^{2k+m_1-2}\,\tilde{s}\overbrace{y_1y_2\cdots y_{a_1+m_2-1}}^{\mbox{ \small{$a_1+m_2-1$ terms}}}
& \mbox{ case (i)}, \\ 
 & \\
\overbrace{y_{k+a_2} y_{k+a_2+1}\cdots y_{k+m_2-1}}^{\mbox{\small{ $m_2-a_2$ terms}}}\overbrace{y_1 y_2\cdots y_{k-1}}^{\mbox{\small{$k-1$ terms}}} 
\tilde{r}\,\overline{y}^{2k+m_1-2}\,
\tilde{s}\overbrace{y_1y_2\cdots y_{a_2-1}}^{\small{\mbox{ $a_2-1$ terms}}}
 & \mbox{ case (ii)}, \\
 & \\
\overbrace{y_{2k+b} y_{2k+b+1}\cdots y_{k+m_2-1}}^{\mbox{\small{ $m_2-k-b$ terms}}}\overbrace{y_1 y_2\cdots y_{k-1}}^{\mbox{\small{$k-1$ terms}}} 
\tilde{r}\,\overline{y}^{2k+m_1-2}\,
\tilde{s}\overbrace{y_1y_2\cdots y_{k+b-1}}^{\small{\mbox{$k+b-1$ terms}}}
 & \mbox{ case (iii)}.
\end{array}
\right.
\]

Accordingly, the terms $\varphi(x^k)$ and 
$\varphi(*)$ are as follows (see Figure~\ref{fig:Q-I-J} for the labeling), abbreviating $r=r_1\cdots r_k$ and $s=s_1\cdots s_k$ as before: 
\[ 
\varphi(x^k)  = 
\left\{ 
\begin{array}{ll} 
\widehat{w_{a_1}}\cdots 
\widehat{w_{1}} \, v_{k+m_1+m_2}  \cdots 
v_{k+m_1+1}\,\widehat{w_{k-1}}
\cdots\widehat{w_{a_1+m_2}}
& \mbox{ in case (i)}, \\
 & \\
v_{k+m_1+a_2}v_{k+m_1+a_2-1}\cdots v_{k+m_1+1}\, 
\widehat{w_{k-1}} \cdots \widehat{w_{a_2}} & \mbox{ in case (ii),} \\
 & \\ 
v_{2k+m_1+b}v_{2k+m_1+b-1}\cdots v_{k+m_1+b+1}\, 
& \mbox{ in case (iii).}
\end{array}
\right.
\]
To write the image of the term (*) under $\varphi$, we use $ZUZ$ 
$z_1\cdots z_{k-1}u_1\cdots u_{m_1} 
z_1\cdots z_{k-1}$ and 
$W$ 
for $w_1w_2\cdots w_{k-1}$ (we {\em only} use these abbreviations in the situation when there are all $k-1$ arrows $z_1,\dots, z_{k-1}$, all arrows $u_1,\dots, u_{m-1}$ and all arrows $w_1,\dots, w_{k-1}$, respectively). 
Then we have: 

\[  
\varphi(*)=
\left\{
\begin{array}{ll}
w_{a_1+1}w_{a_1+2}\cdots w_{k-1}
\,r\,
ZUZ 
\,s w_1\cdots w_{m_2+a_1-1}
& \mbox{ (i)}, \\
 & \\
u_{k+m_2+a_2+1} \cdots u_{k+m_1+m_2}  
W 
r 
ZUZ\, 
s w_1\cdots w_{a_2-1}
& \mbox{ (ii)}, \\
& \\
u_{2k+b+1}\cdots u_{k+m_1+m_2} 
W 
 r 
ZUZ 
s
W\, 
u_{k+m_1+1} u_{k+m_1+2}
\cdots u_{k+m_2+b} & \mbox{ (iii)}.
\end{array}
\right.
\]

For all three cases, the paths $\varphi(x^k)$ and 
$\varphi(*)$ are illustrated in Figure~\ref{fig:3cases-rel3}. 
We first consider case (iii), 
recalling that $\widehat{r}$ and $\widehat{s}$ are the compositions of all $k$ arrows $\widehat{r_i}$ and all $k$ arrows $\widehat{s_i}$, respectively, see Notation~\ref{notation:bridge-paths}.

\vskip.2cm

\noindent
\underline{Case (iii)}

\[
\begin{array}{lcll}
    \varphi(x^k) & = & v_{2k+b-1}v_{2k+b-2}\dots v_{k+b} \\
    & = & u_{2k+b}u_{2k+b+1}\cdots u_{k+b-1} & \mbox{Lemma~\ref{lem:disk bdy}} \\ 
    & = & u_{2k+b-1}\cdots u_{k+m_1+m_2}\widehat{s}\, U \, \widehat{r}u_{k+m_1+1}\cdots u_{k+b-1} & \mbox{Remark~\ref{rem:identifications}} \\ 
\end{array}
\]
By Lemma ~\ref{lem:minimal-paths-in-theta}, $\widehat{s}$ and $WrZ$ are minimal paths between the same boundary vertices in the bridge (cf. Remark~\ref{rem:minimal-3dir}). From Lemma ~\ref{lem:theta-is-thin} they are necessarily equal under $\mathcal{R}_2$. Consequently the above is equal to 
\[
\begin{array}{lcll}   
   \phantom{\varphi(x^k)} &  & u_{2k+b-1}\cdots u_{k+m_1+m_2} 
     W rZ U  \, \widehat{r}u_{k+m_1+1}\cdots u_{k+b+1} & \mbox{Lemma~\ref{lem:minimal-paths-in-theta}} \\
    & = & u_{2k+b-1}\cdots u_{k+m_1+m_2} 
     Wr ZUZ s W u_{k+m_1+1}\cdots u_{k+b+1} & \mbox{Lemma~\ref{lem:minimal-paths-in-theta}} 
\end{array}
\]
as claimed. 

\vskip .2cm

\noindent
\underline{Case (ii)}

We first consider $a_2=1$.
Define $h$ to be the (horizontal) path of length $k-1$ in $\Theta_k$, starting at $(k,k)$, going to $(k,k-1)$, etc., ending at $(k,2)$ (as shown in Figure~\ref{fig:bridge-shape}).  
Recall that $R_2$ is the ideal of relations arising from the potential on $\Theta_k$ and $R_1$ the ideal of relations arising from the potential on $Q_D$. 
\[
\begin{array}{lcll}
    \varphi(x^k) & = & v_{k+m_1+1}\, 
\widehat{w_{k-1}} \cdots \widehat{w_{1}} \\ 
 & = & v_{k+m_1+1}r_k \cdots r_2 h s_2 \cdots s_k & \mbox{using $R_2$} \\ 
 & = & v_{k+m_1+1} v_{m_1+k}\cdots v_{m_1+2}
 h s_2 \cdots s_k & 
 \mbox{Remark~\ref{rem:identifications}} \\
 & = & u_{k+m_1+2} u_{k+m_1+3}\cdots u_{m_1+1} h s_2\cdots s_k & \mbox{using $R_1$ (Lemma~\ref{lem:disk bdy})} \\ 
 & = & u_{k+m_1+2} \cdots u_{k+m_1+m_2}\,\widehat{s}\, 
 u_1\cdots u_{m_1+1}hs_2\cdots s_k & \mbox{Remark~\ref{rem:identifications}} \\ 
 & = & u_{k+m_1+2} \cdots u_{k+m_1+m_2} 
 W rZU u_{m_1+1}hs_2\cdots s_k & \mbox{Lemma~\ref{lem:minimal-paths-in-theta}} 
 \\ 
 & = & u_{k+m_1+2} \cdots u_{k+m_1+m_2} 
 W rZU r_k hs_2\cdots s_k & \mbox{Remark~\ref{rem:identifications}}\\
   & = & u_{k+m_1+2} \cdots u_{k+m_1+m_2} W rZUZs & \mbox{seam relations and $R_2$.}
\end{array}
\] 
This proves (ii) for $a=2_1$. 
If $1<a<k$, we write $d_{[a]}$ for the following (diagonal) path of length $k-a$ in $\Theta_k$: 
$d_{[a]}$ starts at $(a,k)$ and ends at $(1,k-a+1)$, using the arrows 
$(a+1-j,k+1-j)\to (a-j,k-j)$ for $1\le j\le a-1$. 
See Figure~\ref{fig:bridge-shape}. 

\vskip.2cm 

The general situation is then as follows (with $1<a_2<k$): 
\[
\begin{array}{lcll}
    \varphi(x^k) & = & v_{k+m_1+a_2}v_{k+m_1+a_2-1}\cdots v_{k+m_1+1}\, 
\widehat{w_{k-1}} \cdots \widehat{w_{a_2}} \\ 
 & = &  v_{k+m_1+a_2}\cdots v_{k+m_1+1}\, 
r_1\cdots r_{k-a_2}d_{[k-a_2]} & \mbox{using $R_2$} \\ 
 & = &  \underbrace{v_{k+m_1+a_2}\cdots v_{k+m_1+1}\, 
v_{k+m_1} \cdots v_{m_1+a_2+1}}_{k \tiny {\mbox{ terms}}} d_{[k-a_2]} & \mbox{Remark~\ref{rem:identifications}} \\ 
 & = & u_{k+m_1+a_2+1}u_{k+m_1+a_2+2}\cdots u_{m_1+a_2}
 d_{[k-a_2]} & \mbox{using $R_1$} \\ 
 & = & u_{k+m_1+a_2+1} \cdots u_{k+m_1+m_2}\,\widehat{s}\, 
 U u_{m_1+1}\cdots u_{m_1+a_2}d_{[k-a_2]}  \\
 & = &  u_{k+m_1+a_2+1} \cdots u_{k+m_1+m_2}\, Wr Z U u_{m_1+1}\cdots u_{m_1+a_2}d_{[k-a_2]}  \\
 & = & u_{k+m_1+a_2+1} \cdots u_{k+m_1+m_2}\, Wr Z U Z s 
 w_1 \cdots w_{a_2-1} & \mbox{seam and $R_2$,}
\end{array}
\] 
as claimed. 

\vskip .2cm

\noindent 
\underline{Case (i)}
For $1<b<k$, define $e_{[b]}$ to be the path in $\Theta_k$ starting from $(1,b)$, going down via $(2,b)$, $(3,b)$, etc. to $(b,b)$, then left to $(b-1,b)$, etc. to $(2,b)$.

Then $\varphi(x^k)$ is equal to

\[
\begin{array}{cll}
  & \widehat{w_{a_1}}\cdots 
\widehat{w_{1}} \, v_{k+m_1+m_2}  \cdots 
v_{k+m_1+1}\,\widehat{w_{k-1}}
\cdots\widehat{w_{a_1+m_2}} \\ 
 = & \widehat{w_{a_1}}\cdots 
\widehat{w_{1}} \, v_{k+m_1+m_2}  \cdots 
v_{k+m_1+1} r_1\cdots r_{k-a_1-m_2} d_{[a_1]} & \mbox{using $R_2$} \\ 
 = & e_{[a_1]}
 \underbrace{s_{k-a_1+1}\cdots s_k}_{a_1\tiny{\mbox{terms}}} 
\underbrace{v_{k+m_1+m_2}  \cdots 
v_{k+m_1+1}}_{m_2\tiny{\mbox{ terms}}} \underbrace{r_1\cdots r_{k-a_1-m_2}}
_{k-a_1-m_2\tiny{\mbox{ terms}}}d_{[k-a_1-m_2+1]} & \mbox{using $R_2$}\\ 
 = & e_{[a_1]} v_{k+m_1+m_2+a_1} 
 \cdots 
v_{m_1+m_2+a_1+1} d_{[k-a_1-m_2+1]}
 & \mbox{Remark~\ref{rem:identifications}} \\
 = & e_{[a_1]} u_{k+m_1+m_1+a_1+1} \cdots u_{m_1+m_2+a_1} d_{[k-a_1-m_2+1]} & \mbox{using $R_1$} \\ 
 = & e_{[a_1]} \widehat{s_{k-a_1}} 
 \cdots \widehat{s_1} \, U \, \widehat{r_k} \cdots \widehat{r_{k-a_1-m_2-1}}
 d_{[k-a_1-m_2+1]} & \mbox{Remark~\ref{rem:identifications}} \\
 = & w_{a_1+1}w_{a_1+2}\cdots w_{k-1}\, r\,
Z  U \, \widehat{r_k} \cdots \widehat{r_{k-a_1-m_2-1}}
 d_{[k-a_1-m_2+1]} &  \mbox{using $R_2$} \\ 
= & w_{a_1+1}w_{a_1+2}\cdots w_{k-1}\, r\,
Z  U Z\,s w_1\cdots w_{m_2+a_1-1} & \mbox{using $R_2$,}
\end{array}
\]
as claimed. 
\end{proof}

\begin{figure}
    \centering
    \includegraphics[width=10cm]{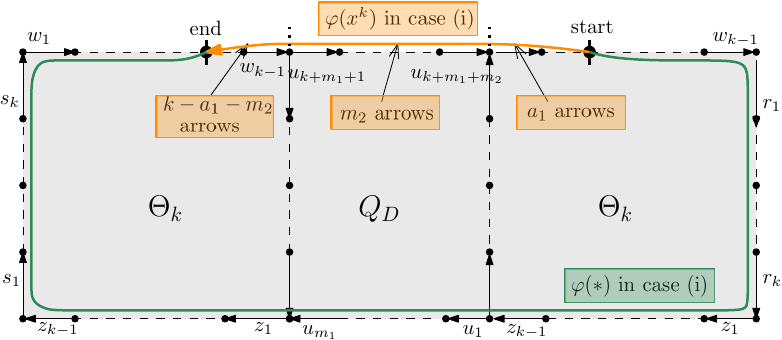}
    \vskip.3cm
    \includegraphics[width=10cm]{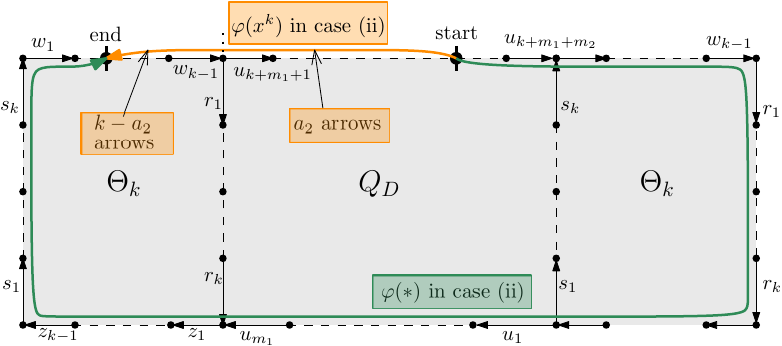}
    \vskip.3cm
    \includegraphics[width=10cm]{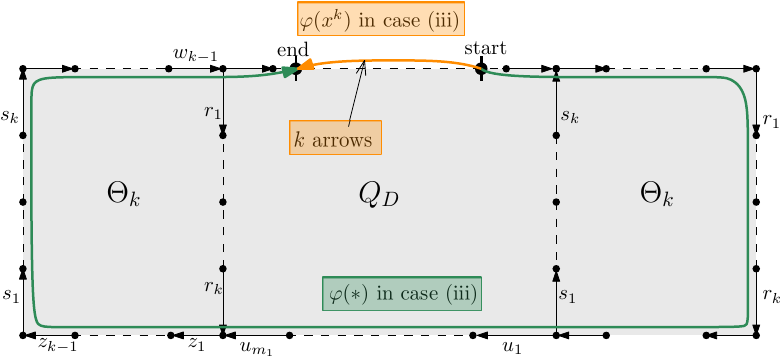}
    \caption{The images of the paths from relation (3) under $\varphi$ on a ``cover'' of the annulus, cases (i), (ii) and (iii)}
    \label{fig:3cases-rel3}
\end{figure}

\begin{Rem}
    By Lemma~\ref{lm:invariance} if $D^\prime$ is a weak Postnikov diagram on the annulus that can be obtained from $D$ by a series of geometric exchange moves, then the boundary algebra of $Q_{D^\prime}$ is isomorphic to $\Lambda_{m_1,m_2,k}$. 
\end{Rem}
\pagebreak

\indent\hspace{.25in} 
\bibliography{essay} 
\bibliographystyle{alpha}

\end{document}